\documentclass[12pt]{amsart}
\usepackage{tikz}
\usetikzlibrary{calc, arrows, decorations.markings} \tikzset{>=latex}
\usepackage{amsmath,amssymb,amsfonts}
\usepackage{amscd, amssymb, latexsym, amsmath, amscd}
\usepackage{amsrefs}
\usepackage[all]{xy}
\usepackage{pb-diagram}
\usepackage{enumerate}
\usepackage{anysize}
\usepackage[sc]{mathpazo} 
\usepackage{enumitem}
\usepackage{tikz}
\usepackage{mathtools}
\marginsize{2.5cm}{2.5cm}{2.5cm}{2.5cm}
\usepackage{graphics}
\usepackage{color}
\DeclareFontFamily{U}{wncy}{}
    \DeclareFontShape{U}{wncy}{m}{n}{<->wncyr10}{}
    \DeclareSymbolFont{mcy}{U}{wncy}{m}{n}
    \DeclareMathSymbol{\Sh}{\mathord}{mcy}{"58}

\marginsize{3cm}{3cm}{3cm}{3cm}
\input xy
\xyoption{all}
\usepackage{pb-diagram}
\usepackage[all]{xy}
\input xy
\xyoption{all}
\theoremstyle{plain}
\newtheorem{thm}{Theorem}

\newtheorem{lemma}{Lemma}
\newtheorem*{lemma4}{Lemma 5}

\newtheorem{corollary}{Corollary}

\newtheorem{definition}{Definition}
\newtheorem{proposition}{Proposition}
\newtheorem*{proposition1}{Proposition 1}

\theoremstyle{definition} \theoremstyle{definition}
\newtheorem{remark}{Remark}
\newtheorem{question}{Question}
\newtheorem{example}{Example}

\theoremstyle{remark}

\newcommand{\Sc}{\mathcal{S}}

\newcommand{\G}{\textsc{\G}}

\newcommand{\g}{\mathfrak{g}}

\newcommand{\uu}{\mathfrak{u}}

\newcommand{\U}{\mathcal{U}}

\newcommand{\Z}{\mathbb{Z}}

\newcommand{\R}{\mathbb{R}}

\newcommand{\Gm}{\mathbb{G}_m}

\newcommand{\Ga}{\mathbb{G}_a}

\newcommand{\C}{\mathbb{C}}

\newcommand{\F}{\mathbb{F}}

\newcommand{\qH}{\mathbb {H}}

\newcommand{\Hom}{{\rm Hom}}

\newcommand{\Ind}{{\rm Ind}}
\newcommand{\ind}{{\rm ind}}

\def\G{{\rm G}}
\def\Aut{{\rm Aut}}
\def\SL{{\rm SL}}

\def\PSL{{\rm PSL}}

\def\Sp{{\rm Sp}}

\def\U{{\rm U}}

\def\Aut{{\rm Aut}}

\def\GL{{\rm GL}}

\def\SO{{\rm SO}}
\def\OO{{\rm O}}
\def\Out{{\rm Out}}

\makeatletter
\let\@wraptoccontribs\wraptoccontribs
\makeatother

\begin{document}

\title[Generic representations for symmetric spaces]
      {Generic representations for symmetric spaces}
      
\author{ Dipendra Prasad}

\contrib[with an appendix by]{Yiannis Sakellaridis}

\subjclass{Primary 11F70; Secondary 22E55}

\begin{abstract}
For a connected quasi-split reductive algebraic group
$G$ over a field $k$, which is either a finite field or
a non-archimedean local field,  $\theta$ an involutive automorphism of
$G$ over $k$, let $K =G^\theta$. Let $K^1=[K^0,K^0]$, the commutator subgroup of $K^0$, the connected component of identity of $K$.
In this paper, we provide a simple condition on $(G,\theta)$ for
there to be an  irreducible admissible generic representation $\pi$ of $G$ 
with $\Hom_{K^1}[\pi,\C] \not = 0$.  The condition is most transparent to state in terms 
of a real reductive group $G_\theta(\R)$ associated to the pair $(G,\theta)$ being quasi-split.
  \end{abstract}

\maketitle
{\hfill \today}

\vspace{4mm}

\tableofcontents

\section{Introduction} \label{parameter}

Let 
${G}$ be a connected reductive algebraic group
over a field $k$, which is either a finite field or
a non-archimedean local field. Let $\theta$ be an involutive automorphism of
$G$ over $k$ and $K =G^\theta$, the subgroup of $G$ on which $\theta$ acts trivially. The pair $(G,\theta)$ or $(G,K)$ is called a symmetric space over $k$.

A well-known question of much current interest is to
spectrally decompose
$L^2(K(k)\backslash G(k))$,
or the related algebraic question to understand
irreducible admissible representations $\pi$ of $G(k)$ with $
\Hom_{K(k)}[\pi,\C] \not = 0$.
Irreducible admissible representations $\pi$ of $G(k)$ with $\Hom_{K(k)}[\pi,\C] \not = 0$ are often called {\it distinguished} representations
of $G(k)$ (with respect to $K(k)$). We refer to the work of
Lusztig \cite{lusztig} for a definitive work for $k$ a finite field, elaborated by Henderson in \cite{Hen} for $G=\GL_n(k)$ and $\U_n(k)$. For $k$ a non-archimedean local   field, there are various works, see e.g. Hakim and Murnaghan in \cite{HM}, Kato and Takano in \cite{KT}. For a general spectral decomposition, see Sakellaridis-Venkatesh \cite{SV}. For $k=\R$, the
spectral decomposition of 
$L^2(K(\R)\backslash G(\R))$ is well understood due to the works of Flensted-Jensen, Oshima et al., see \cite{OM} for a survey. 

Complete results about spectral decomposition of 
$L^2(K(k)\backslash G(k))$, or about
distinguished representations of $G(k)$, will naturally require a full
classification of irreducible admissible representations of $G(k)$. However, for many purposes, less precise, but general results such as multiplicity one property (i.e., $\dim \Hom_{K(k)}[\pi,\C] \leq 1$) when available, are of great importance. As another general question, one might mention the question of
whether there exists a discrete series
representation of $G(k)$ distinguished by $K(k)$, or whether there exists a tempered
representation of $G(k)$ distinguished by $K(k)$. The paper \cite{AGR} of Ash-Ginzburg-Rallis defines a pair $(G,L)$ for a subgroup $L$ of $G$ to be a {\it vanishing pair} if there are no cuspidal representations of $G$ distinguished by $L$, and provides many examples of such pairs without a general criterion about them.

In this paper, we give a general criterion as to when there is a
generic representation of $G(k)$ distinguished by $K(k)$ (assuming of course that $G$ is quasi-split over $k$, a condition which is always satisfied if $k$ is a finite field). Although generic representations are a very special class of representations where the geometric methods of this paper apply,  it appears to us that distinguished generic representations hold key to all distinguished
tempered representations in that the following are equivalent. 

\begin{enumerate}
\item Existence of distinguished tempered representations.
  


\item Existence of distinguished generic representations.
\end{enumerate}

We consider a more precise form of the equivalence above as the (symmetric space)
analogue of Shahidi's conjecture on the existence of a generic representation in a  tempered $L$-packet. 

\begin{question}
Let $(G,\theta)$ be a symmetric space over a local field $k$ with $G$ quasi-split over $k$. Assume that for $K=G^\theta$, $K^0$ the connected component of identity of $K$ is split 
over $k$.
Then if a tempered representation of $G(k)$ is distinguished by $K(k)$, 
then so is some generic member of its $L$-packet? In particular, if there are no generic representations of
$G(k)$ distinguished by $K(k)$, then there are no tempered representations of
$G(k)$ distinguished by $K(k)$?
\end{question}

\begin{remark} We show by an example that the question above will have a negative answer without assuming  $K^0$ 
to be split. For this we  note Corollary 6 of \cite{DP} according to which
for $G=\SO(4,2)$ (where we are using $\SO(p,q)$ to denote the orthogonal group of {\it any} quadratic space of dimension $(p+q)$, 
and rank min$\{p,q\}$), there are
cuspidal representations of $G$ distinguished by $H=\SO(4,1)$ although our main theorem below will show that there are no generic
representations of $\SO(4,2)$ distinguished by $\SO(4,1)$; on the other hand, by \cite{DP}
there are
no cuspidal representations of $G=\SO(4,2)$ 
 distinguished by $H=\SO(3,2)$.  Similarly, for $E/F$ a quadratic extension of non-archimedean local   fields, 
there are examples of distinguished cuspidal representations for $(\Sp_{4n}(F),\Sp_{2n}(E))$ in the paper of Lei Zhang \cite{Zh},
although this paper proves that there no distinguished generic 
representations for $(\Sp_{4n}(F),\Sp_{2n}(E))$. For $(\Sp_{4n}(F),\Sp_{2n}(F)\times \Sp_{2n}(F))$ 
by \cite{AGR}, there  are no distinguished cuspidal representations, and by this paper, there are
no  distinguished generic representations. \end{remark}

Before we come to the statement of the main theorem of this paper, we want to discuss a bit of the {\it universality} of 
reductive groups with involutions over general algebraically closed fields of characteristic not 2. 
Recall that reductive groups have an existence independent of the field over which they are considered,
for example $\Sp_{2n}$ is a reductive group, and we can talk of $\Sp_{2n}(\C)$ as well as
as $\Sp_{2n}(\bar{\F}_p)$ (made precise either by smooth models over open subsets of ${\rm Spec}(\Z)$ containing the point $p$,
or by saying that a reductive group over {\it any} algebraically closed field is given by a root datum). 
In a similar spirit, there is a notion of reductive groups with involution $(G,\theta)$ which makes sense
independent of the algebraically closed field (of characteristic not 2) over which these are defined.

More precisely, for the purposes of this paper, an involution $\theta$ and its conjugates under the  group $G(k)$ play similar role (the action of the  group $G(k)$ being
$(g\cdot \theta)(x) = (g\theta g^{-1})(x)=g\theta(g^{-1}xg)g^{-1}$). If $\Aut (G)(k)[2]$ denotes the set of elements 
$\theta \in \Aut(G)(k)$ with $\theta^2=1$, then the object of interest for this paper is the orbit space,
$$\Aut(G)(k)[2]/G(k).$$

The following (presumably well-known) proposition lies at the basis of this paper and allows one to compare symmetric spaces over different
algebraically closed fields. In this proposition we will use the notion of a {\it quasi-split} symmetric space $(G,\theta)$,
which will play an important role in all of this paper. A symmetric space $(G,\theta)$ over a field $k$ will be said to be quasi-split
if there exists a Borel subgroup $B$ of $G({k})$ such that $B$ and $\theta(B)$ are opposite Borel subgroups of $G$, 
i.e., $B\cap \theta(B)$ is a maximal torus of $G$. Most often in this paper, we will use this concept only over algebraically closed fields
(even if $(G,\theta)$ is defined over a finite or a non-archimedean local   field).

\begin{proposition} \label{inv}
If $\bar{k}_1$ and $\bar{k}_2$ are any two algebraically closed fields of characteristic not 2, then for any connected reductive algebraic group $G$, there 
exists a canonical identification of finite sets 
\begin{eqnarray*}
\Aut(G)(\bar{k}_1)[2]/G(\bar{k}_1) 
& \longleftrightarrow & 
 \Aut(G)(\bar{k}_2)[2]/G(\bar{k}_2).
\end{eqnarray*}
Under this identification of conjugacy classes of involutions, if $\theta_1 \leftrightarrow \theta_2$, 
then in particular,

\begin{enumerate}
\item the connected components of identity 
$(G^{\theta_1})^0$ and $ (G^{\theta_2})^0$ are reductive algebraic groups, and correspond
to each other in the sense defined earlier.
\item the symmetric  space $(G(\bar{k}_1),\theta_1)$ is quasi-split if and only if $(G(\bar{k}_2),\theta_2)$ is quasi-split.
\end{enumerate}
\end{proposition}

We next note the following basic result (called the Cartan classification) about real groups, cf. 
\cite{Se}, Theorem 6 of Chapter III, \S4.

\begin{proposition} \label{Cartan}
  Let $\theta$ be an involutive automorphism of a connected reductive group $G$ over $\C$ with $K=G^\theta(\C)$. Then there exists a naturally associated
  real structure on  $G(\C)$, to be denoted as $G_\theta$, with  $G_\theta(\R)$ invariant under
  $\theta$  on which $\theta$ acts as a Cartan involution, so $G_\theta(\R) \cap K(\C)$ is a maximal compact subgroup of $G_\theta(\R)$. The isomorphism class of the 
real reductive group $G_\theta$ depends only on conjugacy class of the involution $\theta$ as an element in the group ${\rm Aut}(G)$. 
All real
     reductive groups are obtained by this construction (as $\G_\theta(\R)$).
\end{proposition}

\begin{example}
Let $G(\C) = \GL_{m+n}(\C)$, $\theta = \theta_{m,n}$ the involution $g\rightarrow \theta_{m,n}g\theta_{m,n}$ where $\theta_{m,n}$ is the diagonal matrix
  in $\GL_{m+n}(\C)$ with first $m$ entries $1$, and last $n$ entries $-1$. In this case, the real reductive group $\G_{\theta}$ is the group $\U(m,n)$ with maximal compact  $\U(m) \times \U(n)$ which is the compact form
of $G^\theta = \GL_m(\C) \times \GL_n(\C)$.
\end{example}

Here is the main theorem of this paper.

\begin{thm}\label{mainthm1}
  Let $\theta$ be an involutive automorphism of a connected reductive group $G$ over $k$ which is either a finite or a non-archimedean local   field of characteristic not 2 with $K=G^\theta$. 
Then if $G$ is quasi-split over $k$,  and 
if $G(k)$ has a generic representation distinguished by 
$K^1(k)$ for $K^1=[K^0,K^0]$ where $K^0$ is the connected component of identity of $K$,
one of the following equivalent conditions (for $k$ of odd characteristic) 
hold good. 
  \begin{enumerate}
  \item 
    There exists a Borel subgroup $B$ of $G(\bar{k})$ such that
    $B\cap \theta(B)$ is a maximal torus of $G$, i.e., the symmetric space
    $(G,\theta)$ is quasi-split over $\bar{k}$.
  \item Using the identification of groups and involutions over different fields given in Proposition \ref{inv}, suppose $(G,\theta)$ 
over $\bar{k}$ is associated to $(G',\theta')$ over $\C$, then the associated real reductive group $G'_{\theta'}(\R)$ 
is quasi-split over $\R$.
  \end{enumerate}
\end{thm}

\begin{remark}
The equivalence of the two conditions in Theorem \ref{mainthm1}  is part of 
Proposition 1, thus the essence of this theorem is that if $G(k)$ has a generic representation distinguished by $K(k)$, 
then there exists a Borel subgroup $B$ of $G(\bar{k})$ such that $B$ and $\theta(B)$ are opposite Borel subgroups of $G$, 
i.e., $B\cap \theta(B)$ is a maximal torus of $G$, which is what most of this paper does.
\end{remark}

\begin{remark} According to the relative local Langlands conjectures of Sakellaridis-Venkatesh, cf. \cite{SV}, 
$\theta$-quasi-split condition is equivalent to temperedness of the Plancherel decomposition of $G/G^\theta$.
\end{remark}

We summarize the flow of arguments in the paper. The paper tries  to understand when the space of locally constant compactly supported functions   
$\Sc(K(k)\backslash G(k))$
on $K(k)\backslash G(k)$ 
has a Whittaker model (for $k$ a finite or a non-archimedean local   field). 
This involves understanding orbits of $U(k)$, a maximal unipotent subgroup 
in the quasi-split group $G(k)$, on $K(k)\backslash G(k)$. 
The space $\Sc(K(k)\backslash G(k))$
has a Whittaker model if and only if one of the orbits of $U(k)$ on $K(k)\backslash G(k)$ 
supports  a Whittaker model, i.e., there exists an orbit of $U(k)$ with stabilizer say $U_x(k)$ such that the Whittaker functional
$\psi:U(k)\rightarrow \C^\times$ is trivial on $U_x(k)$.
Looking at the action of $U$ on $K\backslash G$ 
over the algebraic closure $\bar{k}$ of $k$, we can hope 
to analyze all possible stabilizers for the action of $U$ on $K\backslash G$ which is what
 the paper 
does to prove Theorem \ref{mainthm1}.  The non-obvious 
assertion which goes in the proof of this theorem is that if $U$ has a certain orbit on $K\backslash G$ 
of the form $U_x\backslash U$ which allows
a Whittaker model, so the Whittaker character is trivial on $U_x$ (notice that such an
orbit of $U$ can be even 1-dimensional), we prove that in fact $U$ must have an orbit on $K\backslash G$ with trivial
stabilizer.  We prove this using detailed structure of reductive groups (over an algebraically closed field) 
which come equipped with an involution, eventually working with root systems with involutions. We know of no
algebraic geometric explanation of why `small' orbits of $U$ on $K\backslash G$ 
of a certain kind forces $U$ to have a large
 orbit on $K\backslash G$, in fact one with no
stabilizer. This is connected with the well-known fact in representation theory that irreducible representations 
of $G(\F_q)$ with Whittaker model are `large', i.e., have dimension of the order of $q^{d(U)}$ where $d(U)$ is the dimension of a 
maximal unipotent subgroup of $G$ assumed to have connected center, cf. Theorem 10.7.7 in \cite{DL}, 
although it is not clear from the definition that this is forced. (One can  prove 
that representations  of $\GL_n(\F_q)$ with Whittaker model are `large' using the structure of the mirabolic subgroup
of $\GL_n(\F_q)$.)

Proposition 1 identifying involutions on different algebraically closed fields of characteristic not 2
is of independent interest, and  is proved in  section \ref{inv2} of the paper. In section \ref{ch2}, we propose a definition of a  symmetric space in characteristic 2, and  suggest that our main theorem of the paper extends
to symmetric spaces in characteristic 2.

Most of the paper is devoted to proving that if there is a generic 
distinguished representation for $(G,\theta)$, the symmetric space must be quasi-split. Section \ref{converse} deals with the converse: if the symmetric space is quasi-split, we prove that there is a generic irreducible representation of $G(k)$
distinguished by $G^\theta(k)$. 

We end the introduction with the following question which we answer in this paper (in the affirmative) 
for symmetric varieties.

\begin{question} Let $G$ be a reductive algebraic group over an algebraically closed field
  $\bar{k}$ operating on a spherical variety $X$. Let $U$ be a maximal unipotent subgroup 
of $G$. Suppose $U$ has an orbit on $X$ of the form $U/U_x$ such that no simple root space of $U$
is contained in $U_x\cdot [U,U]$. Then is there an orbit of $U$ on $X$ with trivial stabilizer? What if we drop the assumption on $X$ to be spherical? By Theorem 10.7.7 of \cite{DL} about dimension of generic representations of $G(\F_q)$ recalled earlier, 
one can deduce 
that $\dim X \geq \dim U$. (The first part of this question in characteristic 0 is now answered by Sakellaridis in the appendix to this paper.) 
\end{question}

\section{Generalities on groups with involution} \label{generalities}

In this section we collect together facts on groups with involutions (over an algebraically closed field $\bar{k}$ of characteristic not 2), all well-known for a long time. We refer to the article \cite{Sp} of Springer as a general reference to this section. (We allow the possibility that the involution is the identity automorphism.)

Given a connected reductive group $G$ with an involution $\theta$ on it, let $K^0$ be the identity component of $K=G^\theta$.
One defines a torus $T$  in $G$ to be $\theta$-split if $\theta(t)=t^{-1}$ for all
$t \in T$. A parabolic $P$ in $G$ is said to be $\theta$-split if $P$ and $\theta(P)$ are opposite parabolics, i.e., $P \cap \theta(P)$ is a Levi subgroup for both $P$ and $\theta(P)$.  

Here are some facts about groups with involutions.

\begin{enumerate}
\item Unless $G$ is a torus or $\theta = 1$, there are always nontrivial $\theta$-split tori in $G$.
\item Maximal $\theta$-split tori in $G$ are conjugate under $K^0$; their common dimension is called 
the rank of the symmetric space $(G,\theta)$.
\item If $Z_G(A)$ (resp $Z_{K^0}(A)$) is the  connected component of identity of the
  centralizer of a maximal $\theta$-split torus $A$ in $G$ (resp. in $K^0$),
  then $Z_G(A)=Z_{K^0}(A)\cdot A$. 
\item Minimal $\theta$-split parabolics in $G$ are conjugate under $K^0$.
\item If $A$ is a maximal $\theta$-split torus in $G$, then its
  centralizer in $G$ is a Levi subgroup for a minimal $\theta$-split parabolic $P$ in $G$.
\end{enumerate}

\begin{definition} (Split and quasi-split symmetric spaces)
A symmetric space $(G,\theta)$ over a field $k$ is said to be split if a maximal $\theta$-split torus
is a maximal torus of $G$.
A symmetric space $(G,\theta)$ is said to be quasi-split if one of the two equivalent conditions hold good:

\begin{enumerate} 
\item For a  maximal $\theta$-split torus $A$, $Z_G(A)$ is a maximal torus of $G$. 
\item There exists a $\theta$-split Borel subgroup in $G$.
\end{enumerate} 
\end{definition}

\begin{remark} \label{dim} If $(G,\theta)$ is a quasi-split symmetric space, 
$A$ a maximal $\theta$-split torus in $G$, $B$ a $\theta$-split 
Borel subgroup in $G$ with $T=Z_G(A)$, a maximal torus contained in $ B$, then in terms of Lie algebras, we have ${\mathfrak g} = {\mathfrak b} + \theta({\mathfrak b} )$, hence ${\mathfrak g} ={\mathfrak g}^\theta+ {\mathfrak b}$. Therefore, for a quasi-split symmetric space $(G,\theta)$,  
$$  \dim (K) = \dim (U) + \dim (T) -\dim(A);$$
in particular, if $(G,\theta)$ is a quasi-split symmetric space, 
 $$  \dim B > \dim (K) \geq  \dim (U) ,$$
which can be improved to the equality:
  $$  \dim (K) =  \dim (U) ,$$
if $(G,\theta)$ is a split symmetric space. Any quasi-split symmetric space is split if $G$ has no outer automorphism. To see a proof, let $B$ be a $\theta$-split Borel subgroup in $G$, i.e.,
$T= B \cap \theta(B)$ is a maximal torus in $G$, which is clearly $\theta$-invariant. Since $G$ has no outer automorphism, and since $\theta$ leaves $T$-invariant, $\theta$ must be conjugation
by an element in the Weyl group $W= N(T)/T$, in fact the longest element in the Weyl group since it takes $B$ to its opposite. However, since $G$ has no outer automorphism, the longest element
in the Weyl group is $-1$, i.e., one which acts on $T$ as: $t\rightarrow t^{-1}$. Thus the symmetric space $(G,\theta)$ is split.
\end{remark}

It is well-known that all the notions and facts above about symmetric spaces $(G,\theta)$ have analogues in the theory of algebraic groups over $\R$ from which they are derived; in particular, if $\bar{k}=\C$, 
then a group $G(\C)$  with involution $\theta$ is split or quasi-split if and only if the real group $G_\theta(\R)$
associated to the pair $(G,\theta)$ due to Cartan (cf. Proposition \ref{Cartan}) 
is split or quasi-split. If the group $G_\theta(\R)$ is split or quasi-split, then clearly the symmetric space $(G,\theta)$ is split 
or quasi-split using the torus $A$  in the Iwasawa decomposition $G_\theta(\R)=KAN$. Conversely, if the symmetric space $(G,\theta)$ is split 
or quasi-split and $A$ is a maximal $\theta$-split torus, then by a well-known theorem of Matsuki in \cite{matsuki}, $A$ has a conjugate by $G^\theta(\C)$ which is defined over $\R$,
and the corresponding torus in $G_\theta(\R)$ is split or quasi-split as is the case for $A$.

Two symmetric spaces $(G,\theta)$ and $(G,\theta')$ are said to be conjugate if 
$\theta$ and $\theta'$, as elements of the group $\Aut(G)$, are conjugate by an element of $G$, 
whereas two symmetric spaces $(G,\theta)$ and $(G,\theta')$ are said to be isomorphic if  
$\theta$ and $\theta'$ are conjugate by an element of $\Aut(G)$.

Just like uniqueness of split and quasi-split groups over any field (with a given splitting field etc.), 
given a reductive group $G$ over $\C$, the set of conjugacy  classes of 
quasi-split symmetric spaces $(G,\theta)$ 
over $\C$ is in bijective correspondence with the set  of 
involutions (i.e., elements of order $\leq 2$) in ${\rm Out}(G) = {\rm Aut}(G)/G$, 
whereas isomorphism classes of quasi-split symmetric spaces $(G,\theta)$ 
over $\C$ is in bijective correspondence with the conjugacy classes  of 
involutions in ${\rm Out}(G) = {\rm Aut}(G)/G$, cf. Theorem 6.14 of \cite{AV}. As an example, for $G=\GL_n(\C)$, there are exactly two 
conjugacy classes as well as isomorphism classes of 
quasi-split symmetric spaces $(G,\theta)$ since ${\rm Out}(\GL_n(\C))=\Z/2$.

The following definition is inspired by the theory of real groups.

\begin{definition}
Let $T=HA$ be a maximal torus for $G$ over a field $k$ left invariant by $\theta$ which operates as identity on $H$ and as 
$t\rightarrow t^{-1}$ on $A$. Then a root space $U_\alpha$ for $T$ is said to be imaginary if $\alpha:T\rightarrow k^\times$
is trivial on $A$ (equivalently, $\theta(\alpha)=\alpha$), 
real if $\alpha$ is trivial on $H$ (equivalently, $\theta(\alpha)=-\alpha$), and complex if it is neither real or imaginary
(equivalently, $\theta(\alpha)\not = \pm \alpha$). If $\alpha$ is imaginary and $U_\alpha \subset G^\theta$, then $\alpha$ is said to be a compact-imaginary root.
\end{definition}

The following well-known lemma is a consequence of Theorem 7.5 of \cite{St} applied to the connected algebraic group
$B \cap \theta(B)$ which contains a maximal torus of $G$. (The last assertion in the lemma is a consequence of 
conjugacy of maximal tori in any connected solvable algebraic group over any field $k$.) 
 
\begin{lemma} \label{tori} For a  symmetric space $(G,\theta)$ over a field  $k$ of characteristic not 2 
and $B$ a  Borel subgroup of $G$,
there exists a maximal torus $T$ of $B$ which is $\theta$-invariant. Further, any two $\theta$-invariant 
maximal tori of $B$ are conjugate under $B^\theta(k)$.
\end{lemma}

The following lemma appears as Lemma 2.4 in \cite{AV2}.

\begin{lemma} \label{trivialonT}
Let $\theta$ be an involution on a reductive group $G$ over 
an algebraically closed field $\bar{k}$. 
If $\theta$ operates
trivially on a maximal torus $T\subset G$, then $\theta$ must be conjugation by an element of $T$.
\end{lemma}
\begin{proof}Since $\theta$ operates trivially on $T$, it operates trivially on the character group of $T$, and hence 
takes any root of $T$ to itself. Therefore, $\theta$ takes all root spaces of $T$ to itself, hence preserves any Borel
subgroup $B$ of $G$ containing $T$. 
Clearly, there is a $t_0 \in T$ 
such that the automorphism $\theta'={\rm Ad}(t_0) \circ \theta $ acts
trivially on all simple root spaces of $T$ in $B$, and hence on $B$. Thus we have an automorphism $\theta'$ of $G$ 
acting trivially on $B$. It is well-known  that such an automorphism of $G$ must be identity since $\theta'(g)\cdot g^{-1}: G \rightarrow G$ descends to give a morphism from $G/B$ to $G$, but there are no non-constant morphisms from a connected projective variety to an affine variety, proving that $\theta'(g)=g$ for all $g \in G$, hence $\theta$ is inner-conjugation by an element of $T$ as desired.
\end{proof}

\begin{lemma} \label{generic-devissage}
For a  symmetric space $(G,\theta)$ over an algebraically closed field $\bar{k}$, let $A$ be a  
$\theta$-split torus in $G$. 
Then $G$ is $\theta$-quasi-split if and only if $Z_G(A) = \{g\in G|gag^{-1}=a, \forall a \in A\}$
which is $\theta$-invariant is $\theta$-quasi-split.
\end{lemma}
\begin{proof} To say that $G$ is $\theta$-quasi-split is equivalent to say that there exists a maximal $\theta$-split 
torus
in $G$, say $A_0$, whose centralizer in $G$ is a maximal torus in $G$. We can assume that $A \subset A_0$, therefore
$A_0\subset Z_G(A)$, and is a maximal $\theta$-split torus in $Z_G(A)$. It follows that if 
the centralizer of $A_0$ in $G$ is a torus, then the centralizer of $A_0$ in $Z_G(A)$ is a torus too. Conversely,
since the centralizer of $A_0$ in $G$ 
is contained in $Z_G(A)$, if 
the centralizer of $A_0$ in $Z_G(A)$ is a torus, 
the centralizer of $A_0$ in $G$ is the same torus. \end{proof}

\section{Whittaker model}
In this section $k$ is a finite or a non-archimedean local   field, $G(k)$ is the group of 
$k$-rational points of a connected quasi-split
reductive group $G$ with $U(k)$ the $k$-rational points of a
fixed maximal connected unipotent group $U$ of $G$, and $\psi:U(k)\rightarrow \C^\times$ a non-degenerate character of $U(k)$, and $B$ the normalizer of $U$ in $G$. By Lemma \ref{tori}, the Borel subgroup $B$ contains a $\theta$-invariant torus $T$ defined over $k$. Whenever we use root spaces in $U$ or in $B$, it is with respect to such a $\theta$-invariant  torus $T$.  It may be noted 
that since tori in $B$ are conjugate under $U(k)$, although the notion of a root space depends on the choice of $T$, but 
the notion of a simple root space in $U/[U,U]$ is independent of the choice of $T$.  We will abuse notation to denote $B,T,U$ also 
for $B(k),T(k),U(k)$. 

For any smooth representation $\pi$ of $G(k)$, $\pi_{U,\psi}$ denotes the largest quotient of $\pi$ on which $U(k)$ operates by the character $\psi$. The representation $\pi$ is said to be {\it generic} if $\pi_{U,\psi} \not = 0$. An important property of the functor $\pi \rightarrow \pi_{U,\psi}$ is that it is exact. In what follows, $\Sc(X)$ denotes the space of compactly supported locally constant functions on a topological space $X$.

\begin{proposition} \label{whittaker} If there exists an irreducible admissible representation $\pi$ of $G(k)$ which is distinguished by $K(k)$ and for which $\pi^\vee_{U,\psi} \not = 0$, then
  $$\Sc(K(k)\backslash G(k))_{U,\psi}\not = 0.$$
\end{proposition}

\begin{proof} By Frobenius reciprocity, the representation $\pi$ of $G(k)$ is distinguished by $K(k)$ if and only if $\pi^\vee$ appears as a quotient of $
  \ind_{K(k)}^{G(k)} \C$:
  $$\Hom_{G(k)}[\ind_{K(k)}^{G(k)} \C, \pi^\vee] \cong \Hom_{G(k)}[\pi, \Ind_{K(k)}^{G(k)} \C] \cong \Hom_{K(k)}[\pi, \C].$$ 

  Since twisted Jacquet functor is exact,
  if $\pi^\vee$ is a quotient of
  $\Sc(K(k)\backslash G(k))$, 
and if  $\pi^\vee_{U,\psi} \not = 0$, then clearly
  $$\Sc(K(k)\backslash G(k))_{U,\psi}\not = 0,$$
proving the proposition.
\end{proof}

\begin{proposition} \label{orbits} With the notation as before,
  $\Sc(K(k)\backslash G(k))_{U,\psi}\not = 0$ 
if and only if there exists an orbit of $U(k)$ on $K(k)\backslash G(k)$
  passing through say $K(k)\cdot x \in K(k)\backslash G(k)$,
  so of the form $U_x(k)\backslash U(k)$ with $U_x(k) = x^{-1}K(k)x\cap U(k)$
  such that $\psi$ is trivial on $U_x(k)\cdot [U,U](k)$.  \end{proposition}

\begin{proof} It is a consequence of Theorem 6.9 of \cite{BZ} that if the group $H(k)$ of $k$-rational points 
of any algbraic group $H$ (over $k$) operates on 
an algebraic  variety $X(k)$ algebraically, and there is a distribution on $X(k)$ supported on a closed subset $X'\subset X(k)$
on which $H(k)$ operates via a character
$m:H(k)\rightarrow \C^\times$, then there must be an orbit of $H(k)$ on $X'\subset X(k)$ which carries 
a distribution   on which $H(k)$ operates via the character
$m:H(k)\rightarrow \C^\times$. (An essential component of this theorem of Bernstein-Zelevinsky is their Theorem A, proved in the appendix to \cite{BZ} for all non-archimedean local fields, that the action of
$H(k)$ on $X(k)$ is always {\it constructible}.)

Therefore if   $\Sc(K(k)\backslash G(k))_{U,\psi}\not = 0$, there is an orbit of $U(k)$ 
on
  $K(k)\backslash G(k) \subset (K\backslash G)(k)$
carrying a distribution on which $U(k)$ operates via $\psi: U(k)\rightarrow \C^\times$ (thus we are applying Theorem 6.9 of \cite{BZ} in the notation above to $H=U$, 
$X=K\backslash G$, $X'=K(k)\backslash G(k) \subset (K\backslash G)(k)$). 

Conversely, if an orbit of $U(k)$ on $K(k)\backslash G(k)$ carries a Whittaker functional, 
$\Sc(K(k)\backslash G(k))_{U,\psi}\not = 0$. This follows because

\begin{enumerate}
 \item $\pi \rightarrow \pi_{U,\psi}$ is an exact functor on the category of smooth representations of $U(k)$.
 \item By Theorem A of Bernstein-Zelevinsky, proved in the appendix to \cite{BZ}, the action of $U(k)$ on $K(k)\backslash G(k)$ is constructible.
\end{enumerate}

If the orbit of $U(k)$
carrying a distribution on which $U(k)$ operates via $\psi: U(k)\rightarrow \C^\times$ is
$U_x(k)\backslash U(k)$, by Frobenius reciprocity 
  $$\Hom_{U(k)}[\ind_{U_x(k)}^{U(k)} \C,
    \psi] \cong \Hom_{U(k)}[\psi^{-1}, \Ind_{U_x(k)}^{U(k)} \C] \cong \Hom_{U_x(k)}[\psi^{-1}, \C].$$
Therefore, for $\ind_{U_x(k)}^{U(k)} \C$ to have Whittaker model for the character $\psi$, $\psi$ must be trivial on $U_x(k)$. But $\psi$ being a character on $U(k)$, it is automatically trivial on $[U,U](k)$, 
proving the proposition.
\end{proof}

\begin{proposition} \label{secondlast}Let $G$ be a quasi-split reductive algebraic group over a field $k$ which is either a finite field or a non-archimedean local field. Let $\theta$ be an involution on $G$ with $K=G^\theta$. 
Then if $\Sc(K(k)\backslash G(k))$ has a Whittaker model there exists  a maximal unipotent subgroup $U$ of 
$G$ such that 
  $U^\theta(k)\cdot [U,U](k)$ contains no simple root space of $U(k)$.  \end{proposition}
\begin{proof} This proposition is a direct consequence of the previous proposition.
\end{proof}

The earlier discussion motivates us to make the following definition.

\begin{definition}(Property $[G]$)
A  symmetric space $(G,\theta)$ over a field $k$ is said to have property $[G]$ ($G$ for {\it generic})
if there exists a Borel subgroup $B$ of $G$ with unipotent radical $U$, and a 
$\theta$-invariant maximal torus $T$ inside $B$ 
such that none of the simple root subgroups of $B$ with respect to $T$ are contained in $U^\theta \cdot [U,U]$.
\end{definition}

\begin{remark} A maximal unipotent subgroup $U(k)$ of $G(k)$ determines the Borel subgroup $B(k)$ as the normalizer 
of $U(k)$. Any two maximal tori in $B(k)$ are conjugate under $U(k)$, and therefore the simple root spaces
for $(U/[U,U])(k)$ are independent of the choice of a  maximal torus in $B(k)$; thus in the definition above of property $[G]$, the choice of $T$ is redundant. A similar remark applies when we talk of root spaces in $U$ in other places in the paper.
\end{remark}

\begin{definition} (Algebraic Whittaker character)
Let $G$ be a quasi-split group over any field $k$ and
$U$ a maximal unipotent subgroup of $G$ defined over $k$.
A homomorphism of algebraic groups $\ell:U\rightarrow {\mathbb G}_a$ defined over $k$ will be said to be an {\it algebraic Whittaker character}
if it is nontrivial restricted to each simple root subspace 
of $U$ (with respect to a maximally split torus $A$ over $k$ normalizing $U$). \end{definition}

An {\it algebraic Whittaker character} over $\bar{k}$ is unique up to conjugacy by 
$T(\bar{k})$. If $k$ is either a finite field or is a local field, and
$\ell$ is defined over $k$, then 
fixing a nontrivial character
$\psi_0:k\rightarrow \C^\times$ allows 
one to construct a  Whittaker character in the usual sense $\psi_0\circ \ell: U(k) \rightarrow \C^\times$.
The map $\ell \rightarrow \psi_0\circ \ell$ is a bijection between algebraic 
Whittaker characters and Whittaker characters on $U(k)$.
Thus, algebraic Whittaker characters are more basic objects being defined over {\it any} field $k$ 
capturing all the attributes of a Whittaker character.

The discussion so far in this section is summarized in the following proposition.

\begin{proposition} \label{final} Let $k$ be either a finite field, or a non-archimedean local  
  field. Let $(G,K)$ be a symmetric space with $G$ quasi-split over $k$. 
If for any point $x \in K(\bar{k}) \backslash G(\bar{k})$, 
  for $U_x(\bar{k}) = x^{-1}K(\bar{k})x\cap U(\bar{k})$,
$U_x(\bar{k}) \cdot [U,U]$,
 contains a simple root space of $U$, then there is no
  irreducible admissible 
representation of $G(k)$ distinguished by $K(k)$ which is generic.   
  \end{proposition}
\begin{proof}Suppose there is an irreducible admissible 
representation of $G(k)$ distinguished by $K(k)$ which is generic.   By Proposition \ref{secondlast}, there exists an orbit of
$U(k)$ passing through a point $x \in K({k})\backslash G({k})$, and 
an abstract Whittaker character  $\ell:U(k)\rightarrow k$ 
such that $\ell$ restricted to $U(k)\cap K(k)$ is trivial, but
$\ell|_{U_\alpha} \not = 0$ for root spaces $U_\alpha$ corresponding to all simple roots $\alpha$ in $U$. Being algebraic (in fact a linear form on a finite dimensional vector space over $k$), such an $\ell$ 
defines an abstract Whittaker character  $\bar{\ell}:U(\bar{k})\rightarrow \bar{k}$ 
such that $\bar{\ell}$ restricted to $U(\bar{k})\cap K(\bar{k})$ 
is trivial (this we prove in Lemma \ref{subtle} below), but
$\bar{\ell}|_{U_\beta} \not = 0$ for each simple root space $U_\beta$ in $U(\bar{k})$ against the
hypothesis in the proposition. 
\end{proof}

In Theorem \ref{mainthm} of the next section we will find  that the geometric condition
on an {\it algebraic Whittaker character} on $U$ being nontrivial on each $U_x$
can be nicely interpreted which will then prove Theorem \ref{mainthm1}.

\begin{remark}Observe that in  Proposition \ref{final} we can deal with all quasi-split groups which become isomorphic over $\bar{k}$ at the same time. For instance, it gives (after we have proved the required statements on the
orbits of $U(\bar{k})$ on $K(\bar{k})\backslash G(\bar{k})$) 
the analogue of Matringe's theorem on non-existence of
  generic representations of $\GL_{m+n}(k)$ distinguished by $\GL_m(k) \times \GL_n(k)$ if $|m-n|>1$ to unitary groups:
  there are no generic representations of 
$\U(V+W)$ (assumed to be quasi-split)
  distinguished by $\U(V) \times \U(W)$ if $|\dim V -\dim W|>1$.
  \end{remark}

\begin{remark}\label{general}
This section is written for a symmetric space $(G,\theta)$ over $k$ together with a given 
unipotent subgroup $U(k)$ of $G(k)$. The involution $\theta$ plays no role in this section, 
and the section remains valid for an arbitrary homogeneous space $K\backslash G$. (The group $U^\theta$ appearing
in this section is then $U\cap K$.) In particular, this section is valid in characteristic 2 except that Lemma \ref{subtle}
uses  the involution $\theta$ crucially and we have not found its analogue in characteristic 2 which will prevent us from
proving the analogue of our main theorem (Theorem \ref{mainthm1}) in characteristic 2 where there is only a subgroup $K$ and not the involution $\theta$. 
\end{remark}

\begin{proposition} \label{final2} Let $k$ be either a finite field, or a non-archimedean local  
  field. Let $(G,K)$ be a symmetric space with $G$ quasi-split over $k$. 
If for any point $x \in K(\bar{k}) \backslash G(\bar{k})$, 
  for $U_x(\bar{k}) = x^{-1}K(\bar{k})x\cap U(\bar{k})$,
$U_x(\bar{k}) \cdot [U,U]$,
 contains a simple root space of $U$, then there is no generic 
  irreducible admissible 
representation of $G(k)$ distinguished by 
$K^1(k)$ for $K^1=[K^0,K^0]$ where $K^0$ is the connected component of identity of $K$.   
  \end{proposition}
\begin{proof} As pointed out in Remark \ref{general}, the considerations in this section also hold good for $K^1\backslash G$.
 The condition ``$U_x(\bar{k}) \cdot [U,U]$ 
 contains a simple root space of $U$,'' is the same for $K^1$ as for $K$. For this we note that the algebraic groups
$ x^{-1}K(\bar{k})x\cap U(\bar{k})$ and $  x^{-1}K^1(\bar{k})x\cap U(\bar{k})$ have the same connected component of identity,
and therefore the condition that ``$U_x(\bar{k}) \cdot [U,U]$ contains a simple root space'' (which is a connected group) is the same for $K$ as for $K^1$.
\end{proof}
 
\begin{lemma} \label{subtle}
With the notation as before, if a linear form $\ell:U(k)\rightarrow k$ 
 is trivial when restricted to $U^\theta(k)=U(k)\cap K(k)$, then $\bar{\ell}:U(\bar{k})\rightarrow \bar{k}$ is also
trivial when restricted to  $U^\theta(\bar{k})= U(\bar{k})\cap K(\bar{k})$.
\end{lemma}
\begin{proof}
Our linear form
$\ell: U(k)\rightarrow k$ arises from a linear form of vector spaces $\ell: (U/[U,U])(k)\rightarrow k$ trivial on 
the image of $U^\theta$ inside $(U/[U,U])(k)$. We will show below
that the image of $U^\theta$ inside $(U/[U,U])(k)$ is a linear subspace of the $k$-vector space 
$(U/[U,U])(k)$ which will prove the lemma.
  The subtlety in the lemma arises from the fact that a priori we only know that $U\cap K = U^\theta$ 
is a unipotent group, and a subgroup of $U$, and in positive characteristic, subgroups of unipotent groups do not have any simple minded structure (even for ${\mathbb G}_a^n \cong U/[U,U]$ which is where our analysis is done). But our unipotent subgroups ($U\cap K$  of $U$) are not the pathological ones as we analyze now.

Let $B$ be the Borel subgroup of $G$ containing $U$, and containing 
a $\theta$-invariant maximal torus $T$ (this is possible by Lemma \ref{tori}).
By generalities around Bruhat decomposition, we know that the 
intersection of any two maximal unipotent subgroups of $G$, in particular $V = U \cap \theta(U)$ 
is a connected unipotent subgroup of $G$ generated by their common root spaces, i.e., 
$$V = U \cap \theta(U) = \prod_{\alpha>0, \theta(\alpha)>0}U_\alpha,$$ 
where the product is taken in any order (it is useful to note that the ``co-ordinates'' corresponding to simple roots are independent of the ordering). 

We will prove that the image of $U^\theta= V^\theta$ under the natural group homomorphism from $V^\theta$ to 
$U/[U,U] \cong \prod_{\alpha \, {\rm simple}}U_\alpha 
\cong {\mathbb G}_a^d$ is a linear subspace, which will prove the lemma. Clearly, 
the image of $V^\theta$ in $\prod_{\alpha \, {\rm simple}}U_\alpha$ is contained in 
$\prod_{\alpha  \in S}U_\alpha$ where $S$ is the set of simple roots $\alpha$ with the property that
$\theta(\alpha)>0$. For simple roots $\alpha$ with $\theta(\alpha)>0$, there are three options:
\begin{enumerate}

\item $\alpha = \theta(\alpha)$ in which case $\theta$ preserves the root space $U_\alpha 
 ={\mathbb G}_a^{d_\alpha}$ 
(for some positive integer $d_\alpha$), and
being an involution, one can decompose  $U_\alpha =  {\mathbb G}_a^{d^+_\alpha} 
+ {\mathbb G}_a^{d^-_\alpha}$ 
with $d_\alpha^+ + d_\alpha^-=d_\alpha$ such that $\theta$ acts as identity on ${\mathbb G}_a^{d^+_\alpha} $
and as $-1$ on ${\mathbb G}_a^{d^-_\alpha}$ (using that we are over a field  of characteristic not 2). 

\item $\alpha \not = \theta(\alpha)$ but both simple. In this case, the image of $V^\theta$ lands inside a linear subspace 
(the diagonal subgroup consisting of element $(u_\alpha,\theta(u_\alpha))$) of 
$U_\alpha \times U_{\theta(\alpha)}$. 

\item $\theta(\alpha)$ is positive but not simple. In this case,  we prove 
that the image of $V^\theta$  inside  $U_\alpha $ is all of $U_\alpha$. 

For this observe that for $u_\alpha \in U_\alpha$, $v=u_\alpha \cdot \theta(u_\alpha)$ belongs to $V$. If $\alpha+\theta(\alpha)$ is not a root, then the
root spaces $U_\alpha$ and $U_{\theta(\alpha)}$ commute, and $\theta(v)=v$, hence it belongs to $V^\theta$ with image 
$u_\alpha \in U_\alpha$,  so the image of $V^\theta$  inside 
$U_\alpha $ is $U_\alpha$. 

Next, assume that $\alpha+\theta(\alpha)$ is a root.
In this case since $\alpha$ and $\theta(\alpha)$ have the same norm, by properties of root systems, the only possible
roots among $i \alpha+ j \theta(\alpha)$ for $i>0,j>0$ is $\alpha+\theta(\alpha)$, hence by Chevalley commutation relation, 
$v\theta(v)^{-1} = [u_\alpha,\theta(u_\alpha)] \in U_{\alpha+\theta(\alpha)}$ 
which is a $\theta$-stable linear
space over a field of characteristic not 2.  
Since
multiplication by 2 is an isomorphism on $U_{\alpha+\theta(\alpha)}$, $H^1(\langle \theta \rangle, U_{\alpha+\theta(\alpha)})=0 $. 
It follows that $v \theta(v)^{-1} = z^{-1}\theta(z)$
for some $z \in U_{\alpha+\theta(\alpha)}$, hence $z u_\alpha \theta(u_\alpha)  \in V^\theta$, proving once again that 
that the image of $V^\theta$  inside  $U_\alpha $ is $U_\alpha$. 
\end{enumerate}

Thus, our linear form
$\ell: U(k)\rightarrow k$ arises from a linear form of vector spaces $\ell: (U/[U,U])(k)\rightarrow k$ trivial on 
the image of $U^\theta$ inside $(U/[U,U])(k)$ which is a linear subspace of the $k$-vector space 
$(U/[U,U])(k)$. The conclusion of the lemma now follows.
\end{proof}
\section{The main theorem}
In this section we work with an arbitrary algebraically closed field $E$ of characteristic not 2. The following proposition is a special case of the main theorem (Theorem \ref{mainthm}) of this section which is proved using this special case. 

\begin{proposition} \label{basic2} Let $G$ be a connected reductive algebraic group over $E$, $T$ a maximal torus in $G$ contained in a Borel subgroup $B$ of $G$. Let $\theta$ be the involution on $G$ which is 
conjugation by an element $t_0 \in T$ which acts by $-1$ on all simple root spaces of $T$ in $B$. Then the symmetric space $(G,\theta)$ is quasi-split, i.e., there exists a 
Borel subgroup $B'$ of $G$ for which  $\theta(B')$ is opposite of $B'$.
\end{proposition}
\begin{proof}
  For the proof of the proposition it suffices to assume that $G$ is an adjoint simple group. In fact if two symmetric spaces 
$(G,\theta)$ and $(G',\theta')$ are related by a homomorphism $\phi: G\rightarrow G'$ with $\theta' \circ \phi = \phi \circ \theta$ such that ${\rm ker}(\phi)$ is central in $G$,
and the image of $G$ under $\phi$ is normal in $G'$, with $G'/\phi(G)$ a torus, then the proposition is true for 
$(G,\phi)$ if and only if it is true for $(G',\phi')$.

The proof of the proposition (for $G$  an adjoint group)  
will be divided into 3 cases. In this case, the element $t_0 \in T$ is unique, and the proposition amounts to saying that the element $t_0\in T$ has a conjugate in 
$G$ which takes $B$ to an opposite Borel which is what we will prove below.

\vspace{2mm}
\noindent{\bf Case 1} (Assuming the Jacobson-Morozov theorem):
 \vspace{2mm}

Let $T_0$ be the diagonal torus in $\SL_2(E)$, $B_0^{\pm}$ the group of upper-triangular and 
lower-triangular matrices in $\SL_2(E)$. 
Assume first that there is a  $j: SL_2(E) \rightarrow G$,  
the Jacobson-Morozov homomorphism corresponding to a regular
unipotent element in $B$ with $j(T_0) \subset T$. 
The Jacobson-Morozov homomorphism is known to exist if either $p=0$, or $p>h$, where $h$ is the Coxeter number of $G$ 
(see \cite{Se2}, Prop. 2, and other references in the bibliography of this paper). 
The element $ j_1=\left  ( \begin{array}{cc}
i &  0  \\
0 & -i 
          \end{array} \right) \in \SL_2(E)$ 
with $i=\sqrt{-1}$, acts by $-1$ on the simple root space of $\SL_2(E)$ contained in $B_0^+$, therefore its image under 
$j$ also acts by $-1$ on all simple root spaces of $T$ in $B$. 

Observe that the matrix
$j_2= \left  ( \begin{array}{cc}
0 &  1  \\
-1 & 0 
          \end{array} \right)$  in $\SL_2(E)$ normalizes $T_0$, and acts on it by $t\rightarrow t^{-1}$.  
It follows that the conjugate of $B$ by $j(j_2)$ is opposite to $B$.

The matrices $j_1 = \left  ( \begin{array}{cc}
i &  0  \\
0 & -i 
          \end{array} \right) $ and 
$j_2= \left  ( \begin{array}{cc}
0 &  1  \\
-1 & 0 
          \end{array} \right)$  in $\SL_2(E)$ are conjugate in $\SL_2(E)$, hence their images under $j$ are conjugate in $G(E)$. Therefore, the proof of the proposition is completed 
whenever we have the Jacobson-Morozov homomorphism corresponding to a regular
unipotent element in $B$.

\vspace{2mm}
\noindent{\bf Case 2} (Classical groups):
For all classical groups $G$ with the standard description of the bilinear form such as $X_1X_{2n}+ \cdots + X_nX_{n+1}$ (for $\Sp_{2n}(E)$ and $\SO_{2n}(E)$) with the standard description of maximal torus and a Borel subgroup containing it as the diagonal subgroup and the upper triangular subgroup,   $\theta$ could be taken to be 
 the involution on $\G(E)$ which is conjugation by the diagonal matrix $t_0$ inside $\GL_n(E)$ where,
$$t_0= \left  ( \begin{array}{cccccccc}
1& {} & {} & {}  &  {}   & {}  &{} &  \\
{}& -1  & {} & {}   &  {}  &  {}  & {} &   \\
{}& {}  & 1  & {}  &   {} & {}   & {}   &\\
{}& {}  &   & -1  &   {} & {}   &{} &  \\
{}& {}  &   & {}  &   \cdot & {}   &{}   &\\
{} &  {}  &   {}   &  &  & \cdot &    & \\
{}& {}  &   & {}  &   {} &  & 1      &    \\
{} & {}   &  {}  & {}  &  & & & (-1)^{n+1}  
\end{array} \right)\cdot
$$

The involution $\theta$ 
preserves the group of upper triangular matrices, but is conjugate to the involution $\theta'$ given by 
conjugation by the anti-diagonal matrix:
$$\left  ( \begin{array}{cccccccc}
{}& {} & {} & {}  &  {}   & {}  &{} & 1  \\
{}&   & {} & {}   &  {}  &  {}  & 1 &   \\
{}& {}  &   & {}  &   {} & 1   & {}   &\\
{}& {}  &   &   &   {} & {}   &{} &  \\
{}& {}  &   & {}     \cdot&  & {}   &{}   &\\
{} &   & {\cdot}     &  &  &  &    & \\
{}& 1  &   & {}  &   {} &  &       &    \\
1 & {}   &  {}  & {}  &  & & &   
\end{array} \right),
$$
for which the group of upper triangular matrices is $\theta'$-split.  We leave the  details to the reader.

\vspace{2mm}
\noindent{\bf Case 3} (Exceptional groups):
\vspace{2mm}

The first observation to make is that if $w_0$ is a longest element in the Weyl group of an adjoint group, then 
by Lemma 5.4 of \cite{AV2}, $w_0$ has a lift to $G$, say $\tilde{w}_0$, with $\tilde{w}^2_0 =1$.

Therefore to prove the proposition for adjoint simple Exceptional group, it suffices to prove that the involutions $t_0$ and 
$\tilde{w}_0$ in $G$ are conjugate. However, for exceptional groups there are very few conjugacy classes of elements of order 2 in $G$ (cf. \cite{Hel}, chapter X, table V for $E=\C$, and therefore also for all 
algebraically closed fields $E$ of characteristic not 2 by what we discuss in section \ref{inv2}):

\begin{enumerate}
\item $G_2$ has only 1;
\item $F_4$ has only 2;
\item $E_6$ has only 2 (with dimension of $G^\theta$ being 38, 46 (EII and EIII in Chapter X, table V of \cite{Hel}), so the quasi-split symmetric space has  $\dim (G^\theta)=38$); note that
the table in Helgason's book has 4 entries for $E_6$, however those involutions which come from inner conjugation action have the property that their fixed point subgroup
has the same rank as the ambient group, eliminating 2 of the 4 entries.
\item $E_7$ has only 3;
\item $E_8$ has only 2.
\end{enumerate}

Since the dimension of the fixed points subgroups for both the involutions $t_0$ and $w_0$ can be easily estimated (see Remark \ref{dim} for $w_0$ which defines a quasi-split symmetric space), 
the proof of the proposition follows.\end{proof}

\begin{remark} At the time of revising this paper, the author found out that Proposition \ref{basic2} is Theorem 6.1 of Springer's paper \cite{Sp3}. \end{remark}

  \begin{example} Let $\theta$ be the involution on $\GL_{n}(E)$ which is conjugation by the diagonal matrix 
$$\left  ( \begin{array}{cccccccc}
1& {} & {} & {}  &  {}   & {}  &{} &  \\
{}& -1  & {} & {}   &  {}  &  {}  & {} &   \\
{}& {}  & 1  & {}  &   {} & {}   & {}   &\\
{}& {}  &   & -1  &   {} & {}   &{} &  \\
{}& {}  &   & {}  &   \cdot & {}   &{}   &\\
{} &  {}  &   {}   &  &  & \cdot &    & \\
{}& {}  &   & {}  &   {} &  & 1      &    \\
{} & {}   &  {}  & {}  &  & & & (-1)^{n+1}  
\end{array} \right)\cdot
$$
The group $G^\theta$ in this case is isomorphic to $\GL_d(E) \times \GL_d(E)$ if $n=2d$ (resp. $\GL_d(E) \times \GL_{d+1}(E)$ if $n = 2d+1$),  and
$G_\theta(\R)=\U(d,d)(\R)$ (resp. $\U(d,d+1)(\R)$)
which is quasi-split over $\R$. The involution $\theta$ 
preserves the group of upper triangular matrices, but is conjugate to the involution $\theta'$ given by 
conjugation by the anti-diagonal matrix:
$$\left  ( \begin{array}{cccccccc}
{}& {} & {} & {}  &  {}   & {}  &{} & 1  \\
{}&   & {} & {}   &  {}  &  {}  & 1 &   \\
{}& {}  &   & {}  &   {} & 1   & {}   &\\
{}& {}  &   &   &   {} & {}   &{} &  \\
{}& {}  &   & {}     \cdot&  & {}   &{}   &\\
{} &   & {\cdot}     &  &  &  &    & \\
{}& 1  &   & {}  &   {} &  &       &    \\
1 & {}   &  {}  & {}  &  & & &   
\end{array} \right),
$$
for which the group of upper triangular matrices is $\theta'$-split.  
\end{example}

The following crucial lemma will be proved in the next section.

\begin{lemma}\label{main}
Let $T=HA$ be a maximal torus for $G$ left invariant by $\theta$ which operates as identity on $H$ and as 
$t\rightarrow t^{-1}$ on $A$. 
Let $B=TU$ be a Borel subgroup of $G$ containing $T$. 
Assume that no simple root of $B$ with respect to $T$ is contained in $U^\theta\cdot [U,U]$. 
Then $\theta$ acts on any  simple root space of $B\cap Z_G(A)$ with respect to $T$ by $-1$.
\end{lemma}

The following  theorem when combined with Proposition \ref{final2} finally proves 
 Theorem \ref{mainthm1} of the introduction.

\begin{thm}\label{mainthm}  Let $(G,\theta)$ be a symmetric space over $E$. 
Suppose $B$ is a Borel subgroup of $G$ with unipotent radical $U$, and $T \subset B$, a 
$\theta$-invariant maximal torus such that none of the simple roots of $B$ with respect to $T$ are contained in $U^\theta \cdot [U,U]$.
 Then the symmetric space $(G,\theta)$ is quasi-split, i.e.,
there exists a Borel subgroup $B'$ of $G$ such that $\theta(B')$ and $B'$ are opposite, i.e., $B' \cap \theta(B')$ is a 
maximal torus of $B'$.
\end{thm}
\begin{proof} 
Assume that $T=HA$ on which $\theta$  operates as identity on $H$ and as 
$t\rightarrow t^{-1}$ on $A$.  

If   rank$(A) =0$, then $T$ is a maximal torus of $G$ on which $\theta$ operates trivially. 
In this case,  we know by Lemma \ref{trivialonT} that such an automorphism of $G$ is an inner-conjugation by an element, say $t_0$, of $T$. We are furthermore  
given that $U^\theta \cdot [U,U]$ has no simple roots of $T$ inside $U$. Since $\theta$ is an involution on $G$ induced by $t_0\in T$, its action on each simple root space of $T$ in $B$ is by $1$ or $-1$. Since $U^\theta \cdot [U,U]$ has no simple roots, we find that $t_0$ operates by $-1$ on all simple roots, and therefore we are in the context
of Proposition \ref{basic2}, which proves the theorem in this case.

If rank$(A) >0$ consider   
the subgroup $Z_G(A)$ of $G$ with the same maximal torus $T=HA $ contained now
in the Borel subgroup $B\cap Z_G(A)$ of $Z_G(A)$. Since $A$ is a central subgroup in $Z_G(A)$, $Z_G(A)/A$ 
has the maximal torus $T/A = H/(H\cap A)$ on which $\theta$ operates trivially. By Lemma \ref{trivialonT}, 
the restriction of $\theta$ to $Z_G(A)/A$ is an inner-conjugation by an element $s_0\in Z_G(A)$. 
Given the hypothesis in this theorem that ``no simple root of $B$ with respect to $T$ is contained in $U^\theta\cdot [U,U]$'', 
by Lemma \ref{main},  $\theta$ acts on any  simple root space of $[B\cap Z_G(A)]/A$ by $-1$. 
Therefore by Proposition \ref{basic2}, $Z_G(A)/A$ --- and therefore $Z_G(A)$ --- is $\theta$-quasi-split.
Now by Lemma \ref{generic-devissage},
$G$ itself is $\theta$-quasi-split, proving the theorem. \end{proof}

We also note the following corollary of this theorem.

\begin{corollary}
Let $(G,\theta)$ be a symmetric space over $E$ with $K=G^\theta$. Then 
there exists a Borel subgroup $B$ of $G$ with unipotent radical $U$, such that
 $U^\theta = U \cap K = \{e\}$  if and only 
if the symmetric space $(G,\theta)$ is quasi-split. 
\end{corollary}
\begin{proof} If the symmetric space $(G,\theta)$ is quasi-split, let $B$ be a Borel
subgroup of $G$ with $B$ and $\theta(B)$ opposite. If $U$ is the unipotent radical of $B$,
clearly $U^\theta = U \cap K \subset U  \cap \theta(U) = \{e\}$, proving one implication in the corollary.

Conversely, assume that 
$U^\theta=1$. Theorem   \ref{mainthm} applies, proving that the symmetric space $(G,\theta)$ is quasi-split. 
\end{proof}
\section{Proof of Lemma \ref{main}}

In this section we work with an arbitrary algebraically closed field $E$ of characteristic not 2 and 
prove Lemma \ref{main} from last section (recalled again here)   
which plays a crucial role in the 
descent argument (from $G$ to $Z_G(A)$) of the previous section. 

The following lemma from Bourbaki \cite{Bou}, Ch. VI, \S1, Corollary 3(a) in Section 6 will play a role
in the proof of Lemma \ref{main} below.

\begin{lemma} \label{connected} Let $R$ be an irreducible root system with $\Delta = \{\alpha_1,\alpha_2,\cdots,\alpha_n\}$ 
a set of simple roots in $R$. For a root $\alpha = \sum_i n_i \alpha_i$ in $R$, let $\Delta(\alpha)$ be the support of $\alpha$
consisting of those simple roots $\alpha_i$ in $\Delta$ for which $n_i\not = 0$. Then $\Delta(\alpha)$ gives rise to a connected subset of the Dynkin diagram of the root system $R$.    
\end{lemma}

\begin{lemma4} \label{basic}
Let $T=HA$ be a maximal torus for $G$ left invariant by $\theta$ which operates as identity on $H$ and as 
$t\rightarrow t^{-1}$ on $A$. 
Let $B=TU$ be a Borel subgroup of $G$ containing $T$. 
Assume that no simple root space of $B$ with respect to $T$ is contained in $U^\theta\cdot [U,U]$. 
Then $\theta$ acts on any  simple root space of $B\cap Z_G(A)$ with respect to $T$ by $-1$.
\end{lemma4}

\begin{proof} 
It suffices to prove the lemma assuming that $G$ is an adjoint group which is 
then a product of simple adjoint groups $G =G_1 \times G_2 \times \cdots \times G_k$. An involution on such a product group $G$ is built 
out of involutions on $G_i$ and involutions on $G_j\times G_j$ which is permuting the two co-ordinates. The lemma is obvious for the latter involution (in fact, in this case, $Z_G(A)=T$, so the lemma is vacuously true), so we are reduced to assuming $G$ to be an adjoint simple group, which we assume is the case in the rest of the proof.

 Observe that a root $\alpha:T\rightarrow E^\times$ for $G$ is a root for $Z_G(A)$ if and only if $\alpha|_A=1$, equivalently, 
$\theta(\alpha)=\alpha$, i.e.,  $\alpha$ is an imaginary root.  Let $\langle X_\alpha \rangle $ be the corresponding root space. If $\theta(\alpha)=\alpha$, then since $\theta$ is an involution, $\theta(X_\alpha)=\pm X_\alpha$. 
If $X_\alpha$ generates a simple root in $B$, then by the hypothesis that 
``no simple root space of $B$ is contained in $U^\theta\cdot [U,U]$'', we find that for imaginary simple roots 
of $B$, $\theta(X_\alpha)=-X_\alpha$. The subtlety in the lemma arises from the fact that simple roots in 
$Z_G(A)$ may not be simple roots in $G$ which is what we deal with in the proof that follows.

Note that under the assumption  that 
``no simple root space of  $B$ is contained in $U^\theta\cdot [U,U]$'',
for any complex simple root $\alpha$, either $\theta(\alpha)<0$ or $\theta(\alpha)$ is simple.  
To prove this, assume the contrary, and let $\theta(\alpha)>0$ and not simple. 
We give a proof of this which is valid in all odd characteristics. For this, begin by observing that $V=U\cap\theta( U)$ is a $\theta$-invariant and $T$-invariant unipotent subgroup of $U$ which is a ``product'' of root spaces
$U_\beta$ such that both $\beta, \theta(\beta)$ are positive. Being $T$-invariant, $V$ is filtered by ${\mathbb G}_a$. 
If $u_\alpha$ is an element of $U$ in the root space $\alpha$, then we claim that:

\begin{enumerate}
\item If  $x=u_\alpha \cdot \theta(u_\alpha)$, then $x$ belongs to $ U^\theta \cdot [U,U]$;

\item $\theta(u_\alpha)$ belongs to $[U,U]$.  
\end{enumerate}

By $(1)$ and $(2)$ above, it follows that $u_{\alpha} \in U^{\theta}\cdot [U,U]$, contrary to our assumption that 
``no simple root space of $B$ is contained in $U^\theta\cdot [U,U]$''.

The proof of $(2)$ is clear since by assumption, $\theta(\alpha)>0$ and not simple. For the proof of $(1)$, note that $x=\theta(x)$ up to $[V,V]$. 
Since $\theta$ is an involution on $V$ preserving $[V,V]$ which is filtered by ${\mathbb G}_a$ on which multiplication by 2 is an isomorphism, hence $H^1(\theta, [V,V])=0$, therefore a $\theta$-invariant in $V/[V,V]$
can be lifted to a $\theta$-invariant  in $V$.

Now, suppose $\alpha$ is a simple root for $B\cap Z_G(A)$ 
but is not a simple root for $B$. 
Let $\Delta= \{\alpha_1,\alpha_2,\cdots,\alpha_n\}$ be the set of simple roots of $T$ in $B$. 
Write $\alpha$ as a sum of simple roots for $G$:
$$\alpha = \sum_in_i \alpha_i.$$
Applying $\theta$ to this equality, we have:
$$\theta(\alpha)=\alpha  = \sum_in_i \theta(\alpha_i).$$
The set $\{\alpha_1,\alpha_2,\cdots,\alpha_n\}$  of simple roots of $T$ on $B$ consists of 
imaginary roots, real roots and complex roots.  The only possibility for a simple root $\alpha_i$ 
which is taken to a positive root under $\theta$ is when $\alpha_i$ is either  imaginary or complex, and in either case if $\theta(\alpha_i)$
is positive, it is simple. Therefore, the only way $\theta(\alpha)=\alpha$, 
the nonzero $n_i$ 
must correspond to either imaginary roots or to pairs of complex simple roots $\{\alpha_i,\theta(\alpha_i)\}$
with equal coefficients. This follows by considering $h(\alpha)$, the height of a root $\alpha = 
\sum_in_i \alpha_i$ defined by $h(\alpha)= \sum_in_i$ and observing that by the remark above  $h(\theta(\alpha)) \leq h(\alpha)$ with equality if and only if the nonzero $n_i$ 
correspond to either imaginary roots or to pairs of complex simple roots $\{\alpha_i,\theta(\alpha_i)\}$.

 Thus we can write any imaginary root as:
$$\alpha = \sum_{i\in I} n_i \alpha_i + 
\sum_{i\in J}n_i [\alpha_i+\theta(\alpha_i)],$$
with $\alpha_i$ simple imaginary roots for $i \in I$, and complex roots for $i\in J$. 

By Lemma \ref{connected}, the support $\Delta(\alpha)$ of  $\alpha$ is a connected subset of the Dynkin diagram of 
the root system $R$ associated to the group $G$. 

Any connected subset of the Dynkin diagram of $G$ is the Dynkin diagram of a reductive subgroup 
of $G$ sharing a maximal torus and with a simple adjoint group. In our case, the Dynkin diagram
associated to $\Delta(\alpha)$ comes equipped with the involution $\theta$ for which the simple imaginary 
roots in $I$ are the fixed points of $\theta$, and simple complex roots have orbits under $\theta$ of 
cardinality 2.

For our goal of proving that   $\theta$ acts on any  simple root space of $B\cap Z_G(A)$ by $-1$, 
a conclusion we have already made for imaginary simple roots, we can assume that for the simple root $\alpha = \sum_{i\in I} n_i \alpha_i + 
\sum_{i\in J}n_i [\alpha_i+\theta(\alpha_i)],$
$\Delta(\alpha)$ has non-empty set of complex roots. Thus the involution $\theta$ on the connected Dynkin
diagram associated to $\Delta(\alpha)$ is non-trivial. Since we know all connected Dynkin diagrams with a non-trivial
involution, here are all the possibilities, together with a check that in each case $\theta(X_\alpha)=-X_\alpha$.

\begin{enumerate}

\item $\Delta(\alpha) = \{\beta_1,\beta_2,\cdots, \beta_{2n} \}$
 is a root system of type $A_{2n}$, with the unique involution on this root system $\theta$ 
(thus the set $I$ in this case is empty), 
and the root $\alpha$ whose support is all of $\Delta(\alpha)$ must be 
$$\alpha = \beta_1+\beta_2+\cdots+ \beta_{2n} 
= \lambda+\theta(\lambda),$$
where $\lambda = \beta_1+\beta_2+\cdots+ \beta_{n}$ which is a root in $A_{2n}$, hence in $G$.

Therefore,
we can assume that 
$$X_\alpha= [X_\lambda, \theta (X_\lambda)] ,$$
for which clearly $$\theta(X_\alpha)=-X_\alpha,$$
proving the lemma. 

\item $\Delta(\alpha) = \{\beta_1,\beta_2,\cdots, \beta_{2n+1} \}$
 is a root system of type $A_{2n+1}$, $n\geq 1$, with the unique involution on this root system $\theta$ 
(thus the set $I$ in this case has exactly one element $\beta_{n+1}$), 
and the root $\alpha$ whose support is all of $\Delta(\alpha)$ must be 
$$\alpha = \beta_1+\beta_2+\cdots+ \beta_{2n+1} 
= \lambda+\beta_{n+1}+ \theta(\lambda),$$
where $\lambda = \beta_1+\beta_2+\cdots+ \beta_{n}$ 
as well as 
$\lambda+\beta_{n+1}$  is a root in $A_{2n+1}$, hence in $G$.
Therefore,
we can assume that 
$$X_\alpha= [[X_\lambda, X_{\beta_{n+1}}], \theta (X_\lambda)] .$$
Using the Jacobi identity, we can write $\theta(X_\alpha)$ as:
\begin{eqnarray*}
\theta(X_\alpha) &= & \,{} \,\,\,[[\theta(X_\lambda), \theta(X_{\beta_{n+1}})], X_\lambda] \\ 
& \stackrel{(*)} = & -[[X_\lambda, \theta(X_{\lambda})], 
\theta(X_{\beta_{n+1}})] -[[\theta(X_{\beta_{n+1}}), X_{\lambda}], \theta(X_\lambda)]
.\end{eqnarray*}

Now we note that $$\lambda+\theta(\lambda)=\beta_1+\beta_2+\cdots+ \beta_{n} + \beta_{n+2}+\beta_{n+3}+\cdots+ \beta_{2n+1},$$
is not a root in $A_{2n+1}$, and hence $[X_\lambda, \theta(X_{\lambda})]=0$. Further, since $\beta_{n+1}$ is a fixed point of $\theta$, it is a simple imaginary root, therefore $\theta(X_{\beta_{n+1}})=
-X_{\beta_{n+1}}$. Therefore, from the equation $(*)$ above,
$$\theta(X_\alpha)=-X_\alpha,$$
proving the lemma in this case. 

 \item $\Delta(\alpha) = \{\beta_1,\beta_2,\cdots, \beta_{n-2}, \beta_{n-1},\beta_{n} \}$
 is a root system of type $D_{n}$, $n\geq 4$, 
with the unique involution on this root system $\theta$, thus the set $I$ in this case is  $I=\{\beta_1,\beta_2,\cdots, \beta_{n-2}\}$, with $\theta(\beta_{n-1})= \beta_n$.
In the standard co-ordinates, one has $\beta_i=e_i-e_{i+1}$ for $i\leq n-1$, and $\beta_n=e_{n-1}+e_n$.
Since the root $\alpha$ has support all of $\Delta(\alpha)$, one can see that the  only 
possible options for  $\alpha$ are:
\begin{eqnarray*}
& & \alpha  =   \beta_1+\beta_2+\cdots+\beta_i +2(\beta_{i+1}+\cdots + \beta_{n-2})+ \beta_{n-1}+ \beta_{n}= e_1+e_{i+1},
\end{eqnarray*}
for some $1< i \leq n-2$.

Note that one of the necessary conditions for two {\it distinct} roots $\{\alpha,\beta\}$ to be simple
 is that $(\alpha,\beta) \leq 0$. Since we are given that the fixed points of the involution i.e., $\beta_j, j \leq n-2$ are simple roots, we must have $(\alpha, \beta_j) \leq 0 $ for all $j \leq n-2$. 
For the possible  $\alpha$ as above, we have $(\alpha,\beta_1)= (e_1+e_{i+1},e_1-e_2) = 1$. So  
$\Delta(\alpha) $
 cannot be  a root system of type $D_{n}$, $n\geq 4$.

\item $\Delta(\alpha)$
 is a root system of type $E_{6}$:
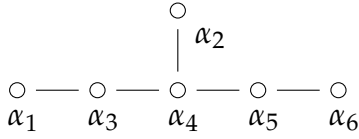
\usetikzlibrary{chains}
\tikzset{node distance=2em, ch/.style={circle,draw,on chain,inner sep=2pt},chj/.style={ch,join},every path/.style={shorten >=4pt,shorten <=4pt},line width=1pt,baseline=-1ex}

\newcommand{\alabel}[1]{%
  \(\alpha_{\mathrlap{#1}}\)
}
\newcommand{\mlabel}[1]{%
  \(#1\)
}
\let\dlabel=\alabel
\let\ulabel=\mlabel
\newcommand{\dnode}[2][chj]{%
\node[#1,label={below:\dlabel{#2}}] {};
}
\newcommand{\dnodea}[3][chj]{%
\dnode[#1,label={above:\ulabel{#2}}]{#3}
}
\newcommand{\dnodeanj}[2]{%
\dnodea[ch]{#1}{#2}
}
\newcommand{\dnodenj}[1]{%
\dnode[ch]{#1}
}
\newcommand{\dnodebr}[1]{%
\node[chj,label={below right:\dlabel{#1}}] {};
}
\newcommand{\dnoder}[2][chj]{%
\node[#1,label={right:\dlabel{#2}}] {};
}
\newcommand{\dydots}{%
\node[chj,draw=none,inner sep=1pt] {\dots};
}

\bigskip
\begin{tikzpicture}
\begin{scope}[start chain]
\foreach \dyni in {1,3,4,5,6} {
\dnode{\dyni}
}
\end{scope}
\begin{scope}[start chain=br going above]
\chainin (chain-3);
\dnodebr{2}
\end{scope}
\end{tikzpicture}
\bigskip

 This also can be ruled out as in case $(3)$ as we argue now.
Following Bourbaki's \cite{Bou}, Plate V on $E_6$, the only positive roots 
of $E_6$ with all coefficients positive, and  for which the coefficients for $\alpha_i$ and $\theta(\alpha_i)$ are the same, are:
\begin{eqnarray*}
\beta_1 & = &\alpha_1+\alpha_2+\alpha_3+\alpha_4+\alpha_5+\alpha_6, \\ 
\beta_2 &=& \alpha_1+\alpha_2+\alpha_3+2\alpha_4+\alpha_5+\alpha_6, \\ 
\beta_3 & = & \alpha_1+\alpha_2 +2 \alpha_3+ 2 \alpha_4+ 2 \alpha_5+\alpha_6, \\ 
\beta_4 & = & \alpha_1+\alpha_2+2\alpha_3+ 3\alpha_4+ 2 \alpha_5+\alpha_6, \\ 
\beta_5 & = & \alpha_1+2\alpha_2+2\alpha_3+3\alpha_4+2\alpha_5+\alpha_6. 
\end{eqnarray*}
The fixed points of the involution are $\alpha_2$ and $\alpha_4$ which are anyway simple roots for $\Delta(\alpha)$. In the standard normalization, each of $\alpha_i$ has $(\alpha_i,\alpha_i)=2$, and nonzero $(\alpha_i,\alpha_j) = -1$. This allows one to make following calculations:
\begin{eqnarray*}
(\beta_1,\alpha_2) & = & 1, \\ 
(\beta_2,\alpha_4) & = & 1, \\ 
(\beta_3,\alpha_4) & = & 1, \\
(\beta_4,\alpha_4) & = & 1, \\
(\beta_5,\alpha_2) & = & 1,
\end{eqnarray*}
proving that $\Delta(\alpha) $
 cannot be  a root system of type $E_6$.
\end{enumerate}\end{proof}

\section{Proof of Proposition 1}\label{inv2}
In this section we give a proof of Proposition 1. We begin by recalling that a connected reductive algebraic group 
$G$ over an algebraically closed field $E$ is  given by a based root datum $\Psi_G = (X,R,S,X^\vee,R^\vee,S^\vee)$ 
where $X,X^\vee$ are 
finitely generated free abelian groups in perfect duality under a bilinear map $X \times X^\vee \rightarrow \Z$, and $R,R^\vee$ 
are finite subsets of $X,X^\vee$ satisfying certain axioms, see e.g. \cite{Sp2}, and $S$ is a set of simple roots in $R$.
This implies that if $G$ is a 
reductive algebraic group over one algebraically closed field $\bar{k}_1$, 
it makes sense to use the same letter $G$ to denote the corresponding    reductive algebraic group over any other
 algebraically closed field $\bar{k}_2$. The group $G$ can be constructed from the based root datum $\Psi_G = (X,R,S,X^\vee,R^\vee,S^\vee)$
in an explicit way starting with the maximal torus $T(E) = X^\vee \otimes E^\times$. This means that the 2-torsion subgroup 
$T(E)[2] = X^\vee \otimes \Z/2 = X^\vee/2X^\vee$, and the set of conjugacy classes of 2-torsion elements in $G$ which is  
$T(E)[2]/W$ where $W=W_G$ is the Weyl group of the group or of the root system.  Thus
the set of conjugacy classes of 2-torsion elements in $G$
depends only on $G$ and is independent of the algebraically closed field $E$ (as long as it is not of characteristic 2). 
Note also that if $\theta$ is an automorphism of $X^\vee$, it gives rise to an automorphism of $T(E)$ for any algebraically closed field $E$.

For any connected reductive group ${ G}$, say over an algebraically closed field $ E$, a {\it pinning}  on
${G}$ is a triple $(B,T,\{X_\alpha \})$ where $T$ is a maximal torus in ${G}$ contained in a Borel subgroup $B$ of ${ G}$, and $X_\alpha$ are nonzero 
elements in the simple root spaces of $T$ contained in $B$. Automorphisms of ${G}$ fixing a pinning will be said to be {\it diagram automorphisms} of ${G}$ (thus for instance if $G$ is a torus, then any automorphism of $G$ will be called a diagram automorphism). If the center of $G$ is $Z$, the group of diagram automorphisms of $G$ is 
isomorphic to ${\rm Out}(G) = {\Aut (G)}/(G/Z)$, which can be read from the root datum, in fact it is isomorphic to $
{\rm Aut}(X,R,S,X^\vee,R^\vee,S^\vee) ={\rm Aut}(X,R,X^\vee,R^\vee)/W_G$  where $W_G$ is the Weyl group of the group or of the root system, 
and which is a normal subgroup of ${\rm Aut}(X,R,X^\vee,R^\vee)$.

We have 
the following short exact sequence of algebraic groups
which is 
split by fixing a pinning on $G$, 
$$1 \rightarrow {\rm Int}({ G}) = G/Z \rightarrow {\Aut}({G}) 
\rightarrow {\Out}({G}) \cong  {\Aut}({ G}, B, T, \{X_\alpha\}) 
\rightarrow 1.$$

We now come to the proof of Proposition 1 of the introduction which we again recall here.

\begin{proposition1} \label{inv}
If $\bar{k}_1$ and $\bar{k}_2$ are any two algebraically closed fields of characteristic not 2, then for any connected reductive algebraic group $G$, there 
exists a canonical identification of finite sets 
\begin{eqnarray*}
\Aut(G)(\bar{k}_1)[2]/G(\bar{k}_1) 
& \longleftrightarrow & 
 \Aut(G)(\bar{k}_2)[2]/G(\bar{k}_2).
\end{eqnarray*}
Under this identification of conjugacy classes of involutions, if $\theta_1 \leftrightarrow \theta_2$, 
then in particular,

\begin{enumerate}
\item the connected components of identity $(G^{\theta_1})^0$ and $ (G^{\theta_2})^0$ are reductive algebraic groups, and correspond
to each other in the sense defined earlier.
\item the symmetric  space $(G(\bar{k}_1),\theta_1)$ is quasi-split if and only if $(G(\bar{k}_2),\theta_2)$ is quasi-split.
\end{enumerate}
\end{proposition1}

\begin{proof} Observe that any element of 
$\Aut(G)(\bar{k}_1)[2]$ 
gives rise to an element of $\Out(G)[2]$, 
so to prove the proposition, it suffices to prove that for 
any $\theta \in   \Out(G)[2]$, 
the set of elements in 
$\Aut(G)(\bar{k}_1)[2]/G(\bar{k}_1) $ giving rise to this $\theta$ is independent of the algebraically closed field $\bar{k}_1$.

If the element $\theta$ is the trivial element, then we are considering elements of $(G/Z)(\bar{k}_1)[2]$ up to conjugation
by $G(\bar{k}_1)$, which is nothing but $(T/Z)(\bar{k}_1)[2]/W_G$. 
Since we are dealing with fields of characteristic $\not= 2$, the structure of $(T/Z)(\bar{k}_1)[2]$ is independent of the algebraically closed field $\bar{k}_1$, and the action of $W_G$ on it also is independent of $\bar{k}_1$, so the proof of the proposition for such elements is completed. The proof in general is a variant of this proof.

Now, take any $\theta \in   \Out(G)[2]$, and let $(G/Z)(\bar{k}_1) \cdot \theta_0$ 
be the set of elements in 
 $\Aut(G)(\bar{k}_1)$ which project to this element $\theta$ in $\Out(G)[2]$ where $\theta_0$ 
is  the unique element of $\Aut(G)$ fixing the 
pinning chosen earlier and giving rise to the element $\theta$ of $\Out(G)[2]$. 

By a theorem due to Gantmacher, cf. \cite{Bo}, elements of order 2 in  $(G/Z)(\bar{k}_1) \cdot \theta_0$ 
can be 
conjugated, using $G(\bar{k}_1)$, 
to lie inside $(T/Z)(\bar{k}_1) \cdot \theta_0$. Another proof of this theorem follows from a well-known theorem of Steinberg, cf. [St, \S7], according to which for any
semi-simple automorphism $\phi$ of a reductive group, there exists a pair $(B,T)$ consisting of a Borel subgroup $B$ and a maximal torus $T$ inside $B$, which is left
invariant under $\phi$.

An element say $t\cdot \theta_0 \in (T/Z)(\bar{k}_1) \cdot \theta_0$ is of order 2
if and only if $(t \cdot \theta_0)^2=t \cdot \theta_0(t)=1$.  
This gives an identification of the set of  elements of order 2 in  $(T/Z)(\bar{k}_1) \cdot \theta_0$ up to 
conjugation by  $T(\bar{k}_1)$ to $H^1(\langle \theta_0 \rangle,  T(\bar{k}_1))$ where 
$\langle \theta_0 \rangle$ is the group of order 2 generated by $\theta_0$. We calculate $H^1(\langle \theta_0 \rangle,  T(\bar{k}_1))$ in the following lemma.

\begin{lemma} \label{H1} Let $\theta_0$ be an involution on the triple $(G,B,T)$, preserving a pinning on simple roots of $T$ in $B$. Then,
 $$H^1(\langle \theta_0 \rangle,  T(\bar{k}_1)) \cong \prod_{\{\alpha| \theta_0(\alpha) = \alpha\}} \Z/2\Z,$$
product taken over the simple roots of $T$ in $B$ left invariant under $\theta_0$.

An  involution of $G$ of the form $t\cdot \theta_0$ acts on all  the simple roots of $T$ in $B$ left invariant under $\theta_0$ by a sign $\pm 1$, giving
us an element in $\prod_{\{\alpha| \theta_0(\alpha) = \alpha\}} \Z/2\Z$. The above identification of involutions of $G$ of the form $t\cdot \theta_0$ with 
$\prod_{\{\alpha| \theta_0(\alpha) = \alpha\}} \Z/2\Z,$ gives rise to a natural map from $\prod_{\{\alpha| \theta_0(\alpha) = \alpha\}} \Z/2\Z$ to itself, 
which is the identity map.

\end{lemma}
\begin{proof} 
 The proof is obvious from the observation that $T/Z (\bar{k}_1)$ is a product of $\Gm$ indexed by the simple roots of $T$ in $B$ which are either permuted or fixed under $\theta_0$;
 in the first case,  $H^1= 0$, and in the second case $H^1 = \Z/2\Z$, completing the proof of the first part of the lemma. The second part of the lemma is clear too.
 \end{proof}

Continuing with the proof of the proposition, we need to consider elements $t\cdot \theta_0 \in (T/Z)(\bar{k}_1) \cdot \theta_0$
not up to conjugation by $(T/Z)(\bar{k}_1)$  
but up to conjugation by  $G(\bar{k}_1)$, which is the same as conjugation up to the Weyl group, or what eventually is needed is up to 
$W^{\theta_0}$, the stabilizer of $\theta_0$ in $W$. Thus, by Lemma \ref{H1},
elements $t\cdot \theta_0 \in (T/Z)(\bar{k}_1) \cdot \theta_0$ of order 2 up to conjugation by  $G(\bar{k}_1)$ is the same as 
$$ \left [ \hspace{2mm} \prod_{\{\alpha| \theta_0(\alpha) = \alpha\}} \Z/2\Z \hspace{2mm} \right ] /W^{\theta_0}.$$

This proves that the set of elements of order 2 in  $(G/Z)(\bar{k}_1) \cdot  \theta_0$ up to 
conjugation by  $G(\bar{k}_1)$ has a structure which is independent of the algebraically closed field $\bar{k}_1$,
proving the canonical identification of finite sets 
\begin{eqnarray*}
\Aut(G)(\bar{k}_1)[2]/G(\bar{k}_1) 
& \longleftrightarrow & 
 \Aut(G)(\bar{k}_2)[2]/G(\bar{k}_2),
\end{eqnarray*}
as desired.

The involutions we have 
constructed above belong to $T(\bar{k}_1) \cdot \theta_0$ with $\theta_0$ a diagram automorphism
of $G$, 
thus these involutions preserve $(T,B)$; for such involutions, Steinberg in Theorem 8.2 of \cite{St} proves that the identity component of $G^{\theta_0}$ is a reductive group, and describes
$G^{\theta_0}$ explicitly in terms of root datum, proving that this group is independent of the algebraically closed field
$\bar{k}_1$. (Actually the above mentioned theorem of Steinberg is proved there only for simply connected groups, but it works well for general reductive groups too 
to describe the connected component of identity of $G^{\theta_0}$.)

Finally, we prove the last assertion in the proposition on quasi-split symmetric spaces using Proposition \ref{A-V} 
 below according to which if $(G,\theta)$ is a quasi-split symmetric space over an algebraically closed field, we can fix a 
$\theta$-stable pair $B\supset T$ (a Borel subgroup and a maximal torus) such that every simple root is 
either complex or noncompact imaginary. If an automorphism of $G$ stabilizes a pair $B\supset T$, it is of the form $t\cdot \theta_0$
as before, and now the fact that for this involution, all simple roots are complex or noncompact imaginary is independent of the field.
 \end{proof}

The point of the following proposition is to describe quasi-splitness of a symmetric space $(G,\theta)$  not in terms 
of the maximal $\theta$-split torus as is usually done, 
but rather in terms of the torus  which is so to say minimal $\theta$-split torus.

\begin{proposition} \label{A-V} A symmetric space $(G,\theta)$  
over an algebraically closed field $E$ is quasi-split if and only if there exists a 
$\theta$-stable pair $B\supset T$ (a Borel subgroup and a maximal torus) such that 
every simple root is either complex or noncompact imaginary; equivalently, for all root spaces $\langle X_\alpha \rangle $ where $\alpha$ is a simple root of $T$ in $B$ with
$\theta(\alpha) = \alpha$, $\theta(X_\alpha) = -X_\alpha$.
\end{proposition}
\begin{proof}
This  proposition is part of the equivalence of parts $(b)$ and $(g)$ of Proposition 6.24 of \cite{AV} proved 
there for complex groups. Here we give an independent  proof  valid for general algebraically closed fields. 

Let's begin by observing that the proposition is true for a symmetric space $(G,\theta)$  if and only if it is 
true for the symmetric space $(G/Z,\theta)$ where $Z$ is the center of $G$ which is naturally left invariant by $\theta$,
inducing an involution on $G/Z$ which is again denoted by $\theta$. Thus for the proof of the proposition, 
we can assume that $G$ is an adjoint group, and hence a 
product of simple groups $G = G_1\times \cdots \times G_d$. Further, we can assume that  $\theta$ either leaves a particular factor $G_i$ invariant, or takes it to another factor, i.e., is the involution $\theta_i: G_i \times G_i \rightarrow G_i \times G_i$ with $\theta_i(x,y)=(y,x)$. 
It is easy to see that the symmetric space $(G_i \times G_i, \theta_i)$ is quasi-split (since it contains the $\theta_i$-split 
Borel subgroup $B_i \times B_i^{\rm op}$), and also has a $\theta_i$-stable 
pair $B_i \times B_i \supset T_i \times T_i$ for which every simple root is complex. Thus the proposition is true for such symmetric spaces.

If $(G_1,\theta_1)$ and $(G_2,\theta_2)$ are symmetric spaces, then clearly the proposition is true for the symmetric spaces 
$(G_1,\theta_1)$ and $(G_2,\theta_2)$ if and only if  the proposition is true for the symmetric space 
$(G_1 \times G_2,\theta_1 \times \theta_2)$. Thus it suffices  to prove the proposition 
for adjoint simple symmetric spaces which is what we assume in the rest of the proof.

If $\theta$ is an inner-automorphism  
preserving $(B,T)$ (for example if the group $G$ has no outer automorphism), then $\theta$ must be
 an inner-conjugation by an element of $T$ which then operates on all simple root spaces of $T$ in $B$ by $-1$ if the condition
``every simple root is either complex or noncompact imaginary'' in the proposition is to be satisfied. By Proposition \ref{basic2}, 
such a symmetric space must be  quasi-split. 

Conversely, if the symmetric space $(G,\theta)$ is quasi-split and $\theta$ is an inner-automorphism,  
let $B$ be a Borel subgroup in $G$ such that $B \cap \theta(B) = T$ is a maximal torus. Then $\theta$ preserves $T$, hence must be the 
inner-conjugation on $G$ by the longest 
element $w_0$ in the
Weyl group of $T$. But it was proved in the course of the proof of Proposition   \ref{basic2}, 
that $w_0$ and $t_0$ (the element of $T$ which acts by $-1$ on all simple roots in $B$) are conjugate in $G$, i.e., $w_0 = gt_0g^{-1}$. It is then easy to see that $gBg^{-1}$ is invariant under $\theta$, on which
$\theta$ operates by conjugation by $gt_0g^{-1}$ which acts by $-1$ on all simple root spaces.

Thus the proposition is proved for all  symmetric spaces $(G,\theta)$ where $G$ is an 
adjoint simple groups, and the  involution $\theta$ is an inner automorphism, 
leaving for us to deal with non-inner involutions on 
$A_n,D_n,E_6$.

Observe too that one part of the proposition that if, ``every simple root is either complex or noncompact imaginary,'' then $(G,\theta)$ is quasi-split is part of Theorem \ref{mainthm}. 
This is also part of the conclusion of Springer in \cite{Sp3}, \S6, \S7.
We need to prove the converse, i.e., if $(G,\theta)$ is a quasi-split symmetric space, then there is a choice of $(B,T)$ invariant under $\theta$ such that
every simple root is either complex or noncompact imaginary. (Note that we are considering $\theta$ which is an  outer automorphism.)

The case of $A_n, D_n$ is easily done using their explicit descriptions as classical groups. For $G=E_6$, there are 2 conjugacy classes of outer involutions. 
One of which can be taken to be the diagram automorphism $\theta_0$ whose fixed point subgroup is $F_4$, and the other involution can be taken 
to be $\theta= t_0\cdot \theta_0$ where $t_0$, or it is the same as $\theta$,  operates on all simple roots left invariant by $\theta_0$ by $-1$. We have clearly
two distinct conjugacy classes of outer automorphism in $\theta_0$  and $\theta$, and since there are only two, these are the two. It suffices then to note that 
$(E_6, F_4)$ is not quasi-split.
But a quasi-split symmetric space has dimension $\geq U$, which in our case is 36, whereas $\dim(E_6/F_4)$ is 26, so $E_6/F_4$ is not quasi-split.
  \end{proof}
\section{Symmetric spaces in characteristic 2} \label{ch2}
In most literature on symmetric spaces $(G,\theta)$, it is traditional to assume that 
one is dealing with fields of characteristic not 2, for example the article of Springer \cite{Sp}, as well as the article of Lusztig \cite{lusztig} assumes this is the case. 
For applications to representation theory, one would prefer not to make this assumption. For instance, 
a basic example of a symmetric space is $G(k) = \GL_{m+n}(k)$  with the involution $\theta: g\rightarrow \theta_{m,n}g\theta_{m,n}$ where $\theta_{m,n}$ is the diagonal matrix
  in $\GL_{m+n}(k)$ with the first $m$ entries $1$, and the last $n$ entries $-1$. 
In this case, the fixed points of the involution is 
$K(k)=\GL_{m}(k) \times \GL_n(k)$, which makes good sense in all characteristics even though the involution itself does not in characteristic 2. In fact, as is well-known, the subgroup $G^\theta$ of $G$ uniquely determines $\theta$ in characteristic zero, 
say by a Lie algebra argument, thus using fixed points of an involution seems a good enough  replacement for the involution 
itself till we realize we have lost the main anchor for the arguments with symmetric spaces.

Thus the first order of business is to define what's meant by a symmetric space in characteristic 2, which we now take as a pair $(G,K)$ with $K$ a `symmetric' subgroup of $G$,
 which we will presently define. 
It may be remarked that one reason for circumspection regarding involutions (and therefore symmetric spaces) 
in characteristic 2 is that their fixed point subgroups need not be reductive. Our definition below will continue to assume 
reductiveness for the subgroup although it  seems to be useful not to insist on it so as not to exclude the interesting  
example of Shalika subgroup (centralizer of the unipotent element  
$u=\left  ( \begin{array}{cc}
{I_n}&  I_n  \\
0& {I_n} 
\end{array} \right)$ 
of order 2 
inside $\GL_{2n}(F)$ where $I_n$ is the $n \times n$ identity matrix).

\begin{definition} (a)(Symmetric subgroup) Let $(G,K)$ be a pair, consisting of a 
connected reductive algebraic group $G$, and a connected reductive subgroup $K$ of $G$ over a field $k$. 
The subgroup $K$ of $G$ 
is said to be a symmetric subgroup of $G$ if 
one can spread these to split reductive group schemes $K_R\subset G_R$ where $R$ is a discrete valuation ring with residue field $R/{\mathfrak m}=k$, 
and quotient field $L$ (of characteristic 0), 
such that $(G_L,K_L)$ is a symmetric space
in the usual sense (over a field of characteristic zero now) defined by an involution on $G_R$.

(b) (Symmetric space) A pair of reductive groups $(G,K)$ over a field $k$, 
with $K$ a symmetric subgroup over ${k} $ as defined in $(a)$, 
will be said to be a symmetric space over $k$. 

(c) (Quasi-split symmetric space) A symmetric space $(G,K)$ over a field $k$ of characteristic 2 will be said to be quasi-split if the corresponding symmetric space 
$(G_L,K_L)$ over $L$ is quasi-split. 
\end{definition}

\begin{proposition}
With the notation as in the definition above, if a symmetric space $(G,K)$ over 
$k$, a finite or non-archimedean local field,
has property $[G]$, i.e., there exists a Borel subgroup $B$ of $G$ with unipotent radical $U$ such that 
$(U\cap K) \cdot [U,U]$ 
contains no simple root,
then the corresponding symmetric space  over $L$ also has this property, 
hence is quasi-split by Theorem \ref{mainthm}.  
\end{proposition}
\begin{proof} It may be conceptually simpler to fix $U$, and vary $K$ (up to conjugacy by $G(k)$), allowing us to use 
a fixed unipotent group scheme $U_R$ and a maximal torus $T_R$ normalizing $U_R$ together with its root system. 
Since $(U\cap K) \cdot [U,U]$ is the reduction modulo $\mathfrak m$ of the  unipotent group $(U_R\cap K_R) \cdot [U_R,U_R]$ in
$G_R$, this is clear. 
\end{proof}

\begin{remark}
 In the absense of Lemma \ref{subtle} in charactaristic 2, we are not able to prove that 
Theorem \ref{mainthm1} remains valid for symmetric spaces over 
$k$, a finite or non-archimedean local field, in characteristic 2;
in more detail, we are not able to prove that if $G$ is quasi-split over $k$,  and 
if $G(k)$ has a generic representation distinguished by 
$K^1(k)$ for $K^1=[K^0,K^0]$ where $K^0$ is the connected component of identity of $K$,
then the symmetric space $(G,K)$ is quasi-split. 
\end{remark}

\section{Examples}
This paper was conceived to explain many examples 
of  symmetric spaces $(G,\theta)$ for which it was known that there are no generic representations of $G(k)$ 
distinguished by ${G^\theta(k)}$. Usually such 
theorems are known only in the context of $G=\GL_n$, and proved  by different methods (Gelfand pairs, mirabolic subgroups, theory of derivatives, ...). Here are some of these examples, all consequence of our
Theorem \ref{mainthm1}.

\begin{enumerate}
 \item  
Let $G(k) = \GL_{m+n}(k)$, $\theta = \theta_{m,n}$ the involution $g\rightarrow \theta_{m,n}g\theta_{m,n}$ where $\theta_{m,n}$ is the diagonal matrix
  in $\GL_{m+n}(k)$ with first $m$ entries $1$, and last $n$ entries $-1$. In this case, the real reductive group $\G_{\theta}$ is the group $\U(m,n)$ which  is quasi-split over $\R$ if and only if $|m-n| \leq 1$. It is a theorem due to Matringe,
Theorem 3.2 in \cite{matringe},  that if
  there is a generic 
representation of $\GL_{m+n}(k)$ distinguished by $\GL_m(k) \times GL_n(k)$ then $|m-n| \leq 1$.

\item   Let $G(k) = \GL_{2n}(k)$, $\theta$ the involution on $\GL_{2n}(k)$ given by 
$g\rightarrow J{}^t g ^{-1} J^{-1}$ with $J$ any skew-symmetric matrix in $\GL_{2n}(k)$.  The fixed point set of $\theta$ 
is the symplectic group $\Sp_{2n}(k)$. In this case, the real reductive group $\G_{\theta}$ is the group $\GL_n(\qH)$ where $\qH$ is the quaternion division algebra over $\R$. The real reductive group $\GL_n(\qH)$ is not quasi-split, and it is known that there are
no generic 
representations of $\GL_{2n}(k)$ distinguished by $\Sp_{2n}(k)$, a theorem due to Heumos-Rallis \cite{HR}.

\item   Let $G(k) = \GL_{n}(k)$, $\theta$ the involution on $\GL_{n}(k)$ given by 
$g\rightarrow J{}^t g ^{-1} J^{-1}$ with $J$ any symmetric matrix in $\GL_{n}(k)$.  The fixed point set of $\theta$ 
is the orthogonal group $\OO_{n}(k)$. In this case, the real reductive group $\G_{\theta}$ is the group $\GL_n(\R)$ which is  split.
It is known that there are generic 
representations of $\GL_{n}(k)$ distinguished by $\OO_{n}(k)$.

\vspace{2mm}

Twisted analogues of examples in $(1),(2),(3)$, such as:

\vspace{2mm}

\item $\U(V+W)$ 
  containing $\U(V) \times \U(W)$ as fixed point of an involution. 
The group $G_\theta$ is the same as in $(1)$, so not quasi-split if $|\dim V -\dim W|>1$; 
we are not certain if one knew before that there are no  generic representations of $\U(V+W)$ (assumed to be quasi-split)
  distinguished by  $\U(V) \times \U(W)$ if $|\dim V -\dim W|>1$.

\item $\U(V\otimes E)$ is a quasi-split unitary group over $F$ with $E/F$ quadratic, $V$ a symplectic space over $F$ with $V\otimes E$ the corresponding skew-hermitian space over $E$, then $\Sp(V) \subset \U(V\otimes E)$, and analogous to the work of Heumos-Rallis \cite{HR}, there are no generic representations
of $\U(V\otimes E)$ distinguished by $\Sp(V)$. This result is due to \cite{Kum}, Theorem 4.4.

Here are the  other classical groups.

\item $\SO(V+W)$ 
  containing ${\rm S}[\OO(V) \times \OO(W)]$ as the fixed points of an involution. Assume that $\dim V=m$, $\dim W=n$. 
Then the  group $G_\theta = \SO(m,n)(\R)$ which  is  not quasi-split if $|\dim V -\dim W|>2$. 
We are not certain if one knew before that there are no  generic representations of $\SO(V+W)$ (assumed to be quasi-split)
  distinguished by  ${\rm S}[\OO(V) \times \OO(W)]$ if $|\dim V -\dim W|>2$.

\item $\Sp(V+W)$ 
containing $\Sp(V) \times \Sp(W)$ as the fixed points of an involution. Assume that $\dim V=m$, $\dim W=n$. 
In this case the  group $G_\theta = \Sp(m,n)(\R)$ which  is  never quasi-split. 
We are not certain if one knew before that 
there are no  generic representations of $\Sp(V+W)$ distinguished by  $\Sp(V) \times \Sp(W)$. These are vanishing pairs of \cite{AGR}, i.e., there are no  cuspidal representations of $\Sp(V+W)$ distinguished by  $\Sp(V) \times \Sp(W)$. 

There is also the twisted analogue of this example $G=\Sp_{4n}(F) \supset \Sp_{2n}(E)=H$. In this case also,
by our theorem,  there are no
generic representations of $G$ distinguished by $H$.

\item Besides the involutions used in the previous two examples, there is another kind of involution for the groups $\SO_{2n}(F)$ as well as $\Sp_{2n}(F)$. For this suppose $V=X+X^\vee$ is a complete polarization on a $2n$-dimensional vector space over 
$F$ which may be orthogonal or symplectic. Define $j_0$ to be the involution on the corresponding classical group 
   $G(V)$ obtained by inner conjugation of the element $j \in \GL(V)(F)$ which acts on $X$ by multiplication by $i=\sqrt{-1}$, and
on $X^\vee$ by multiplication by $-i$. These involutions define $G_\theta=\SO^{\star}(2n,\R), \Sp(2n,\R)$ 
in the orthogonal and 
symplectic cases respectively with maximal compact $\U(n,\R)$ in both cases. The group $  \SO^{\star}(2n)$ 
is not quasi-split, 
whereas $\Sp(2n,\R)$ is. Therefore, by our theorem, there are no
 generic representations of 
$\SO_{2n}(F)$ which are distinguished by $\GL_n(F)$, whereas there are generic representations of 
$\Sp_{2n}(F)$ distinguished by $\GL_n(F)$ (Theorem \ref{mainthm1} only rules out generic distinguished representations and does not construct one when it allows one which we take up in the next section). We are not certain if one knew before that 
there are no  generic representations of $G(V)=\SO_{2n}(F)$ distinguished by  $\GL_n(F)$. It is a  vanishing pair of \cite{AGR} (they assert  this only for $n$ odd whereas we do not distinguish between $n$ even and $n$ odd).

\end{enumerate}

\section{The converse and concluding remarks}\label{converse}
In this paper we have proved that if 
for a symmetric space 
$(G,\theta)$ over a finite or a non-archimedean local   field $k$, there is an irreducible generic representation
of $G(k)$ distinguished by $G^\theta(k)$, then the symmetric space over $\bar{k}$ is quasi-split. It is
worth emphasizing  that this theorem used 
knowledge of $G$ as well as $\theta$ only over $\bar{k}$.

If the symmetric space is quasi-split over $k$, we
have the following converse. The author thanks Matringe for his help with the proof 
of this proposition especially by  suggesting the use of Proposition 7.2 of \cite{offen}
in this proof.

\begin{proposition}
Let $(G,\theta)$ be a symmetric space over a finite or a non-archimedean local   field $k$ which is quasi-split over $k$, thus there is a Borel subgroup $B$ of $G$ over $k$ 
 with $B\cap \theta(B) = T$,  a maximal torus of $G$ over $k$. If $k$ is finite, assume that its cardinality is large enough (for a given $G$). Then    there is an irreducible generic unitary principal series representation
of $G(k)$ distinguished by $G^\theta(k)$.
\end{proposition}

\begin{proof}
It is easy to see in the non-archimedean local   case by a well--known Lie algebra argument that $K(k)\cdot B(k)$ is an open subset of 
$G(k)$ giving rise to an open subset $H(k)\backslash K(k) \subset B(k)\backslash G(k)$ 
where $H(k) = K(k)\cap B(k)$ is the $k$-points of a torus inside $K$ contained in the maximal torus
$T = B\cap \theta(B) $ of $B$. Thus 
we have the inclusion $\Sc(H(k)\backslash K(k)) \subset \Sc(B(k)\backslash G(k))$ 
of compactly supported functions. If $Ps(\chi)$ is the principal series representation of $G(k)$ induced using a character $\chi:T(k)\rightarrow \C^\times$ for which $\chi|_{H(k)}=1$, 
then have the inclusion   
$\Sc(H(k)\backslash K(k)) \subset Ps(\chi)$; 
it is  important at this point to note 
that this inclusion remains true for normalized induction because of  lemma \ref{modulus} below.

By Frobenius reciprocity, the subspace $\Sc(H(k)\backslash K(k)) \subset Ps(\chi)$ carries an $H(k)$-invariant linear form 
if  $\chi|_{H(k)}=1$.
 It is a result of Blanc-Delorme, cf. Theorem 2.8 of \cite{BD}, see Proposition 7.2 of \cite{offen} for a precise statement in our specific context, 
 that if $\chi|_{H(k)}=1$, then indeed
the principal series representation $Ps(\chi)$ is distinguished. 

It suffices then to construct characters $\chi$ of $T(k)$ with $\chi|_{H(k)}=1$ such that 
the principal series representation $Ps(\chi)$ is irreducible. This can be done for $k$ a finite field 
under the hypothesis that it has large enough cardinality (for a given $G$) by choosing a regular character of
$(T/H)(k)$, i.e., one for which $\chi^w\not = \chi$ for any $w\not = 1$ in  $W_G$.

For $k$ a non-archimdean local field,  note that 
by Lemma \ref{modulus} if $\delta_B$ is the modulus function of $B$, then $\delta_B$ restricted to $H(k)$ is trivial.
Therefore, for any nonzero real number $t$, $\chi=\delta_B^{it}$ is a unitary character of $T$, trivial on $H(k)$, 
and is not invariant by any nontrivial element of the Weyl group of $G$. By well-known results of Bruhat,
such principal series representations of $G$ are irreducible, completing the proof of the proposition.
\end{proof}
\begin{lemma} \label{modulus}
With the notation as above, if $\delta_B$ is the modulus function of $B$, then $\delta_B$ restricted to $H(k)$ is trivial.
\end{lemma} 
\begin{proof}Observe that $\delta_{\theta(B)}(\theta(t)) = \delta_B(t)$. 
Therefore if $\theta(t)=t$ as is the case for
elements in $H(k)$, we have $\delta_{\theta(B)}(t) = \delta_B(t)$. 
On the other hand, $\theta(B)$ being the opposite Borel, $ \delta_{\theta(B)}(t) = \delta_B(t^{-1})$. This completes the proof of the lemma. \end{proof}

Our paper has dealt with quasi-split symmetric spaces, 
but if the symmetric space $(G,\theta)$ is not quasi-split, question arises as to  what are the `largest'
representations which contribute to the spectral decomoposition of $L^2(K(k)\backslash G(k))$? There is a natural  
$\SL_2(\C)$ inside the $L$-group of $G$ which controls this, and which can be constructed as follows. Let $T=HA$ be a 
maximal torus in $G$ on which $\theta$ operates by identity on $H$, and $A$ is a maximal $\theta$-split torus in $G$. 
The centralizer $M=Z_G(A)$ of $A$ in $G$ is a Levi subgroup of $G$, hence it associates  a Levi subgroup $\widehat{M}$ in the dual group
$\widehat{G}$ too. The $\SL_2(\C)$ associated to the regular unipotent conjugacy class in $\widehat{M}$ plays an important role for 
$L^2(K(k)\backslash G(k))$ whose spectral analysis is dictated by the centralizer of this (Arthur) $\SL_2(\C)$ in $\widehat{G}$. 
For example, for the symmetric space $(\GL_{2n}(k),\Sp_{2n}(k))$, $M=\GL_2(k)^n$, and we will be looking at 
the representation of $\SL_2(\C)$ inside $\GL_{2n}(\C)$ which is $[2]+ \cdots +[2]$ where $[2]$ is the standard 2-dimensional representation of $\SL_2(\C)$, with centralizer $\GL_n(\C)$ which is well-known to control the 
spectral decomoposition of $L^2(\Sp_{2n}(k)\backslash \GL_{2n}(k))$.  

Our previous analysis done in the presence of a $\theta$-split Borel subgroup can be done with a minimal $\theta$-split
parabolic $P=MN$ with $M=Z_G(A) = Z_K(A)\cdot A$, and now involve principal series representations $Ps(\pi)$ induced from
an irreducible representation $\pi$ of $M$ but which are trivial on $Z_K(A)$ (which contains the derived subgroup of $M$) to
describe $K(k)$-distinguished (principal series) representations of $G$. Because for 
these principal series representations $Ps(\pi)$, $\pi$ is trivial on the derived subgroup of $M$, 
 $\pi$ is one dimensional, giving some substance to the suggestion in the previous paragraph.

\vspace{5mm}

\noindent{\bf Acknowledgement:} The author thanks the University of Maryland for hosting his stay 
in an academically stimulating environment 
in the fall of 2017 
where this work was started. He thanks J. Adams for many inspiring discussions. He thanks Y. Sakellaridis for providing the Appendix to this paper, proving 
a more general form of Theorem \ref{mainthm} in characteristic 0.
He also thanks
M. Borovoi, M. Brion, N. Matringe,  Arnab Mitra,   Sandeep Varma for helpful discussions. Thanks are also due to the referee for a careful reading and useful feedback.
Author's research was supported by the JC Bose Fellowship of DST, India. This work was supported by a grant of
the Government of the Russian Federation
for the state support of scientific research carried out
under the  agreement 14.W03.31.0030 dated 15.02.2018.

\newpage 

\appendix{}

\section{An application of a theorem of Knop} 
\vspace{.5cm}

\centerline{\bf By }
\vspace{.5cm}

\centerline{\bf Yiannis Sakellaridis}
\vspace{.5cm}

The goal of this appendix is to apply a theorem of Friedrich Knop in order to give another proof of Theorem \ref{mainthm} of this paper, which applies to a more general class of $G$-varieties, when the characteristic of the field is zero.

More precisely, let $G$ be a reductive group over a field $k$ in characteristic zero, and let $X$ be a normal variety with a $G$-action. For the result of this appendix, there is no harm in generality in assuming that $k$ is algebraically closed, and we will do so. Let $B$ be a Borel subgroup of $G$, with unipotent radical $U$. To formulate the result, we recall that the isomorphism type of a ``generic $B$-orbit'' is well-defined, namely: there is a parabolic $P\supset B$, with a Levi subgroup $L$ and a torus quotient $L\to A_X$, such that $X$ contains a $P$-stable, Zariski open subset which is $P$-isomorphic to
 $$ (A_X \times^L P) \times V,$$
 where $V$ is a variety with trivial $P$-action, see \cite[Satz 2.3 and Korollar 2.4]{KnWeyl}.
 
For any $G$-variety $X$, we will denote the above parabolic $P$ attached to its normalization by $P(X)$; its conjugacy class does not depend on any choices. For a symmetric space defined by an involution $\theta$ of $G$ it is well-known, and easy to see, that $P(X)$ is the class of \emph{minimal $\theta$-split parabolics}.

If $P^-$ is the opposite parabolic with $P^-\cap P = L$, and $S$ denotes the kernel of the composition of maps:
$$ P^-\to L \to A_X,$$
Knop calls the variety $X':=S\backslash G$ the ``horospherical type'' of $X$.

In this appendix we prove:

\begin{thm}\label{appthm}
 Let $G$ be a reductive group over a field $k$ in characteristic zero, and let $X$ be a variety with a $G$-action. Let $B\supset U$ be a Borel subgroup with its unipotent radical, and $x\in X$ any point. If there is a generic character $U\to \Ga$ (that is, one which is nonzero on every simple root subgroup) which is trivial on the stabilizer $U_x$, then $P(X)=B$.
\end{thm}

The proof is a direct application of the following theorem of Knop: To state it, assume that $X$ is smooth, and recall that the infinitesimal action of the Lie algebra $\g$ induces vector fields on $X$; the adjoint to this is the moment map 
$$ m_X : T^*X\to \g^*.$$
We use similar notation for the horospherical type,
$$ m_{X'}: T^*X' = \mathfrak s^\perp\times^S G\to \g^*.$$

Knop proves in \cite[Satz 5.4]{KnWeyl}:
\begin{thm}[Knop] \label{Knoptheorem}
 The closures of the images of $m_X$ and $m_{X'}$ in $\g^*$ coincide.
\end{thm}

Let us consider an additive character $\psi: U\to \Ga$, and its differential $d\psi:\uu\to \Ga$. The character is non-trivial on every simple root subspace if and only if its differential is, and in this case they will both be called ``generic''. Let us introduce similar language for elements of the dual Lie algebra $\g^*$: We will call an element $v\in \g^*$ \emph{generic} if there is a Borel subgroup $B\supset U$ as above such that the image of $v$ under the restriction map $\g^*\to \uu^*$ is generic. 

To use Theorem \ref{Knoptheorem}, observe first:
\begin{lemma}\label{thelemma}
 The image of $m_{X'}$ is closed, and it contains generic vectors if and only if $P(X)=B$.
\end{lemma}

\begin{proof}
The map $\mathfrak s^\perp\times^S G \to \g^*$ factors through $\mathfrak s^\perp\times^{P^-} G$ (with $P^-$ as in the definition of $S$), and since $\mathfrak s^\perp\subset \g^*$ is closed, the map  $\mathfrak s^\perp\times^{P^-} G\to \g^*$ is proper, hence has closed image.

Notice that $X'$ lives over the flag variety $P^-\backslash G$. Let $v\in T^*X'$, which up to the $G$-action we can assume to belong to $\mathfrak s^\perp$. If $B\supset U$ is a Borel subgroup such that the image of $v$ in $\uu^*$ is generic, then $\uu$ cannot intersect the subalgebra $[\mathfrak p^-,\mathfrak p^-]\subset \mathfrak s$, where $v$ is trivial. But this is only possible if $P^-$ is opposite to $B$, that is, if $P(X)=B$.
\end{proof}

\begin{proof}[Proof of Theorem \ref{appthm}]
Without loss of generality, we may assume $X$ to be smooth. Indeed, given the existence of equivariant resolutions of singularities (see \cite[Proposition 3.9.1]{Kollar}), if the theorem is true for an equivariant resolution $\tilde X\to X$, it is true \emph{a fortiori} for $X$.

Let $B\supset U$ be a Borel subgroup, $x\in X$, and consider the moment map with respect to the $U$-action, restricted to the fiber over $x$:
 $$ T_x^*X \to \g^* \to \uu^*.$$
Its image is equal to $\uu_x^\perp$, where $\uu_x$ is the Lie algebra of the stabilizer of $x$ in $U$. Hence, if there is a generic character $U\to \Ga$ which is trivial on the stabilizer $U_x$, or equivalently: such that its differential $\uu\to \Ga$ is trivial on $\uu_x$, the image of $T_x^*X$ under the moment map contains a generic vector, and by  Lemma \ref{thelemma} this implies that $P(X)=B$. 
\end{proof}

\vspace{.5cm}

\vspace{.5cm}

\noindent{Dipendra Prasad,}

\noindent  Indian Institute of Technology Bombay, Mumbai. 

\noindent Tata Institute of Fundamental Research, Mumbai.

\noindent St Petersburg State University, St Petersburg.

\noindent Email: {\tt prasad.dipendra@gmail.com}
    
    \vspace{.5cm}
    
    \noindent{Yiannis Sakellaridis,}
    
\noindent{Rutgers University, Newark, NJ 07102.}
    

\begin{bibdiv}

  \begin{biblist}\bib{AV}{article}{
author={J. Adams}
author={D. Vogan}
title={L-groups, projective representations, and the Langlands classification}
journal={ Amer. J. Math.}
volume={ 114, no. 1}
pages={ 45-138}
year= {1992}}

\bib{AV2}{article}{
author={J. Adams}
author={D. Vogan}
title={Contragredient representations and characterizing the local Langlands correspondence.}
journal={ Amer. J. Math.}
volume={ 138, no. 3}
pages={ 657-682}
year= {2016}}


\bib{AGR}{article}{
author={A. Ash}
author={D.Ginzburg}
author={S.Rallis}
title={Vanishing periods of cusp forms over modular symbols}
journal={Math. Ann}
date={1993}
 volume={296, no. 4}
pages={ 709-723}}



\bib{BZ}{article}{
  author={J. Bernstein},
  author={A. Zelevinskii},
  title={Representations of the group $\GL(n,F)$  where $F$ is a non-archimdean local field}, 
  journal={  Russian Math. Surveys },
  volume={31(3)},
  date={1976},
   pages={1-68},
  }

\bib{BD}{article}{
author={P. Blanc}
author={P. Delorme}
title={ Vecteurs distributions H-invariants de représentations induites, pour un espace sym\'etrique 
r\'eductif $p$-adique G/H}
journal={ Ann. Inst. Fourier (Grenoble)}
volume={ 58 (1)}
year={ 2008}
pages={213--261}}


\bib{Bo}{article}{ 
author= {A.  Borel} 
title={Automorphic L-functions}
journal= {Automorphic forms, representations and L-functions, 
Proc. Sympos. Pure Math., Oregon State Univ., Corvallis, Ore., 1977, Part 2, 
Proc. Sympos. Pure Math., XXXIII, Amer. Math. Soc., Providence, R.I.}
pages={27-61} 
date={1979}}

\bib{Bou}{article}{ 
author={N. Bourbaki}
title={ Lie groups and Lie algebras. Chapters 4-6. Translated from the 1968 French original by Andrew Pressley. Elements of Mathematics}
journal={Springer-Verlag, Berlin}
year={ 2002}} 

\bib{DL}{article}{ 
author={P. Deligne}
author={G. Lusztig}
title={ Representations of reductive groups over finite fields}
journal={ Ann. of Math. (2)}
volume={ 103, no. 1}
year={1976}
pages={ 103-161}} 


\bib{DP}{article}{ 
author={S. Dijols}
author={D. Prasad}
title={Symplectic models for Unitary groups}
journal={arXiv:1611.01621, Transactions of the AMS, DOI: https://doi.org/10.1090/tran/7651. }
}
 

\bib{HM}{article}{
author={J. Hakim}
author={F. Murnaghan}
title={Distinguished tame supercuspidal representations}
journal={Int. Math. Res. Pap. IMRP}

date={ 2008}

volume={ no. 2}}

\bib{Hel}{article}{
author={S. Helgason}
title={ Differential geometry, Lie groups, and symmetric spaces.}
journal={ Pure and Applied Mathematics, 80. Academic Press, Inc.}
year={1978} 
}

\bib{Hen}{article}{
author={A. Henderson}
title={Symmetric subgroup invariants in irreducible representations of $G^F$, when $G=\GL_n$ }
journal={  J. Algebra }
date={2003}
volume={ 261}
}

\bib{HR}{article}{
author={M. Heumos}
author={S.Rallis}
title={Symplectic-Whittaker models for $\GL_n$}
journal={  Pacific J. Math. }
date={1990}
volume={ 146}
pages={247-279}
}


\bib{KT}{article}{
author={S. Kato}
author={K. Takano}
title={ Subrepresentation theorem for $p$-adic symmetric spaces.}

journal={Int. Math. Res. Not. IMRN}
date={2008, no. 11}
}


\bib{Kum}{article}{
author={R. Kumanduri}
title={Distinguished representations for unitary groups}
 journal={Pacific J. Math}
volume={178, no. 2}
date={1997}
pages={293-306}
}



\bib{lusztig}{article}{
author={G. Lusztig}
title={Symmetric Spaces over a Finite Field}
journal={ In The Grothendieck Festschrift, edited by P. Cartier et. al., Boston, Birkhäuser}
volume={ 3}

pages={ 57--81}
 
date={1990}
}

    
\bib{matringe}{article}{
author={N. Matringe}
title ={
Linear and Shalika local periods for the mirabolic group, and some consequences } 
journal ={J. Number Theory}
volume = {138}
date = {2014}
pages={1-19}}

\bib{matsuki}{article}{
author={T. Matsuki} 
title={The orbits of affine symmetric spaces under the action of minimal parabolic subgroups,}
 journal={J. Math. Soc. Japan,}
 volume= {31}
 date={1979}
 pages={331-357.}
 }

 \bib{offen}{article}{
 author={O. Offen}
title={On parabolic induction associated with a $p$-adic symmetric space}
journal={J. Number Theory}
volume={170}
year={2017}
pages={211-227}
}

\bib{OM}{article}{
author={T. Oshima}
author={T. Matsuki}
title ={A description of discrete series for semisimple symmetric spaces. }
journal ={Group representations and systems of differential equations, Adv. Stud. Pure Math.,}
volume = {4}
date = {1982}
pages={331-390}}

\bib{SV}{article}{
author={Y. Sakellaridis}
author={A. Venkatesh}
title={Periods and harmonic analysis on spherical varieties}
volume={396}
journal={Asterisque}
date={2017}
}




\bib{Se}{article}{
author={ J.-P. Serre}
title={ Galois cohomology} 
journal={Translated from the French by Patrick Ion and revised by the author. Corrected reprint of the 1997 
English edition. Springer Monographs in Mathematics. Springer-Verlag}
date={ 2002} 
}

\bib{Se2}{article}{
author={ J.-P. Serre}
title={Exemples de plongements des groupes $\PSL_2(\F_p)$ dans des groupes de Lie simples.}
journal={Invent. Math.}
volume={124, no. 1-3,}
pages={ 525-562}
year={1996}
}


\bib{Sp2}{article}{
author={T. Springer}
title={Linear algebraic groups}
journal={Progress in Mathematics, Birkhauser}
volume={9}
date={1998}
}

\bib{Sp}{article}{
author={T. Springer}
title={Some results on algebraic groups with involutions}
journal={Algebraic groups and related topics (Kyoto/Nagoya, 1983), 525--543, Adv. Stud. Pure Math., 6, North-Holland, Amsterdam}
date={1985}
}

\bib{Sp3}{article}{
 author={T. Springer}
 title={The classification of involutions of simple algebraic groups.}
journal={J. Fac. Sci. Univ. Tokyo Sect. IA Math.}
volume={34 (1987), no. 3,}
pages={655-670} 
}

\bib{St}{article}{
author={R. Steinberg}
title={Endomorphisms of linear algebraic groups,}
journal={Memoirs of the American Mathematical Society, No. 80 American Mathematical Society, Providence, R.I.}
date={1968}}


\bib{Zh}{article}{
author={L.Zhang}
title={Distinguished tame supercuspidal representations of symmetric pairs $(\Sp_{4n}(F),\Sp_{2n}(E))$, 
With an appendix by Dihua Jiang and the author}
journal={Manuscripta Math}
volume={148, no. 1-2}
date={ 2015}, 
pages={213-233}} 


\end{biblist}


\end{bibdiv}

\begin{bibdiv}

\begin{biblist}

\bib{Kollar}{book}{
AUTHOR = {Koll\'{a}r, J\'{a}nos},
     TITLE = {Lectures on resolution of singularities},
    SERIES = {Annals of Mathematics Studies},
    VOLUME = {166},
 PUBLISHER = {Princeton University Press, Princeton, NJ},
      YEAR = {2007},
     PAGES = {vi+208},
      ISBN = {978-0-691-12923-5; 0-691-12923-1},
}

  
\bib{KnWeyl}{article}{
AUTHOR = {F.~Knop},
TITLE = {Weylgruppe und {M}omentabbildung},
JOURNAL = {Invent. Math.},
VOLUME = {99},
YEAR = {1990},
NUMBER = {1},
PAGES = {1--23},
}






\end{biblist}
\end{bibdiv}
    \end{document}